\documentclass[12pt]{article}
\usepackage{amssymb,amsmath,amsfonts,amsthm}
\title{Groups of generalized Moufang type and $\Z_2$-graded algebras.}
\author{Ilya Gorshkov }
\date{ }
\newcounter{parag}

\usepackage{amssymb}
\usepackage{amsfonts}
\usepackage{amsmath}
\usepackage{enumerate}

\newcommand\F{\mathbb{F}}

\newcommand\Z{\mathbb{Z}}

\newcommand\cF{\mathcal{F}}

\newcommand\al{\alpha}
\newcommand\bt{\beta}
\newcommand\gam{\gamma}
\newcommand\dl{\delta}

\newcommand\lm{\lambda}

\newcommand\la{\langle}
\newcommand\ra{\rangle}
\newcommand\lla{\la\!\la}
\newcommand\rra{\ra\!\ra}

\newtheorem{lem}{Lemma}
\newtheorem{Th}{Theorem}
\newtheorem{defi}{Definition}
\newtheorem{prop}{Proposition}
\newtheorem{question}{Question}
\newtheorem{rem}{Remark}
\newtheorem{cor}{Corollary}

\begin{document}
\maketitle
\textit{Abstract.}
 A pair $(G,T)$ is called a faithful odd transposition group if $T$ is a normal set of involutions generating the group $G$ and the product of any two distinct elements of $T$ has odd order. We introduce a special subclass of such groups, a \emph{generalized Moufang group of $p$-type} (or $GM(p)$-type), in which the product of any two distinct involutions from $T$ has a fixed prime order $p$. For any such group $(G,T)$ and a scalar parameter $\eta$ in a field $\F$, we construct a non-associative, non-commutative algebra $A = A_{\F}(G,T,\eta)$.

We prove that every element of $T$ considered as an element of the algebra $A$, is a primitive semisimple idempotent, defining a $\Z_{2}$-grading of $A$. The Miyamoto group of $A$ with respect to $T$ is isomorphic to $G/Z(G)$. The algebra $A$ contains no nontrivial right ideals and, for a specific choice of the parameter $\eta$, admits a symmetric left Frobenius form.

When $G$ is a free Burnside group of odd prime period \(p\) extended by an involutory automorphism, the finiteness of $G$ is equivalent to the finite-dimensionality of $A_{\F}(G,T,\eta)$,  providing a reformulation of the Burnside problem. For $p=5$ and $\eta=-1/3$, the algebra generated by two idempotents from $T$ is left-axial and satisfies the Monster-type fusion law $\mathcal{M}(4/3, -4/3)$. For a prime $p>5$, the two-generated algebra is also axial, but obeys a more general fusion law.

Although the algebra $A_{\F}(G,T,\eta)$ is initially defined using a group $GM(p)$-type, we show that it admits an intrinsic, group-free characterization by axiomatizing a class of so-called $GM(p,\eta)$-type algebras. We prove that every algebra in this class is isomorphic to one arising from the construction above, establishing the equivalence of the two definitions.

\section*{Introduction}
Let $G$	be a group.	If $G$ is generated by a normal set of involutions $T$ such that the order of the product of any two elements from $T$ does not exceed $n$, then $(G,T)$ is said to be an $n$-transposition group and elements of $T$ are called $n$-involution.
	Fischer \cite{Fi64} first studied $3$-transposition groups in the special case where the product of any two distinct $3$-transpositions has order $3$, such groups were called Moufang type groups. He showed that a finite group of Moufang type is soluble and has a $3$-group of index $2$. In the course of studying simple $3$-transposition groups, Fischer \cite{Fi71} found 3 simple sporadic groups $Fi_{22}$, $Fi_{23}$, $Fi_{24}$. Cuypers and Hall \cite{CH95} explored the classification program and subsequently obtained a classification of finitely generated $3$-transposition groups. Their result implies that $3$-transposition groups are locally finite. It is known that the sporadic group Baby Monster is a $4$-transposition group. The largest sporadic group Monster is a $6$-transposition group and is generated by just three involutions.
	
	Let $G$ be a finite group and $\omega$ be a set of positive integers. A pair $(G, T)$ is called an $\omega$-transposition group if $T$ is a normal set of involutions of $G$, $G= \la T\ra$, and	$|uv|\in \omega $ for any choice of distinct $u$ and $v$ in $T$.
	Aschbacher \cite{Ash2} introduced the concept of odd transposition groups. If $\omega$ is the set of all positive odd integers, then $\omega\cup\{2\}$-transposition group $(G, T)$ is called an odd transposition group. Aschbacher described the odd transposition groups with trivial solvable radical \cite{Ash2}.
	
	In this paper, we introduce the concept of faithful odd transposition groups. If $\omega$ is the set of all positive odd integers, then $\omega$-transposition group $(G, T)$ is called an faithful odd transposition group. 
	Let us define a special subclass of odd transposition groups. Faithful odd transposition group $(G,T)$ is called a generalized Moufang group of type $n$ (briefly $GM(n)$-type) if the order of the product of any two elements of $T$ is $n$.
	
	Using Glauberman's $Z^*$-Theorem we get the following statement describing finite faithful odd transposition groups. 
	
\begin{Th}
	Let $(G,T)$ be a finite faithful odd transposition group. Then $G\simeq K(G)\rtimes \la t\ra$.
\end{Th}

	The free Burnside group of rank $d$ and exponent $n$, denoted $B(d, n)$, is defined as the quotient group of the free group $F_d$ on $d$ generators and $N$ is the normal closure of the set $\{g^n \mid g \in F_d\}$ in $F_d$. In 1902, Burnside \cite{B02} formulated the conjecture: the groups $B(d, n)$ are finite. He himself proved that the groups $B(d,2)$ and $B(d,3)$ are finite for all $d$. In \cite{Sa40}, Sanov proved the finiteness of the group $B(d,4)$ for all $d$. Hall \cite{Ha58} proved the finiteness of the group $B(d,6)$ for all $d$.
	
	Novikov and Adian \cite{NA68} showed that the groups $B(d,n)$ are infinite for all odd $n\geq4361$ and $d$ greater than $1$. This result has been generalized many times. Atkarskaya, Rips and Tent \cite{R24} proved that $B(d, n)$ is infinite for all odd $n\geq 557$. Lysenok \cite{Ly} proved the infinity of groups $B(d, n)$ for $n\geq 16k$ where $k > 8000$ and $d$ greater than $1$.
	
	Let $G=\la x_1,....,x_d\ra$ be a free Burnside group of odd prime period $p$. Then $G$ has an automorphism $t$ of order two which inverts the free generators $x_1,...,x_d$. Note that the group $G.\la t\ra$ is generated by a set $T=\{t, x_1t,..., x_dt\}$ of involutions. It is also easy to see that $t$ and $x_it$ are conjugate in the dihedral group $\la t, x_i t\ra$. Thus $(G.\la t\ra, T)$ is a group of $GM(p)$-type. 
\begin{Th}
 Let $(G,T)$ be an odd transposition group and let $G$ have a bounded period. If $G/Z(G)$ is a finite group and $T$ is a finite set, then $G$ is a finite group.
\end{Th}	
	The theory of axial algebras represents a contemporary approach that reveals profound connections between group theory and non-associative algebras. Its origins lie in the construction of the Griess algebra, whose automorphism group is the sporadic Monster group.
	
	Axial algebras were first introduced by \cite{HRS}. An axial algebra is a commutative but non-associative algebra generated by a special set of idempotents called axes. For each axis $a $, the multiplication operator $\text{ad}_a$ is semisimple, i.e. allowing the algebra to be decomposed into a direct sum of its eigenspaces.
	
	The multiplication structure in the $\cF$-axial algebra $A$ is controlled by a fusion law $\cF$. If the fusion law $\cF$ is $\mathbb{Z}_2$-graded, then each axis $a$ can be associated with $\mathbb{Z}_2$-grading of $A$ and an $A$ automorphism $\tau_a$ of order $2$, known as the Miyamoto automorphism. For a given generating set of axes $X$, the group generated by the corresponding Miyamoto automorphisms is called the Miyamoto group.

	Rowen and Segev \cite{RS} study non-commutative axial algebras generated by primitive idempotents whose left and right multiplication operators are semisimple and possess at most three distinct eigenvalues.
	
	%Современным инструментом, связывающим теорию групп с неассоциативной алгеброй, является \textbf{теория аксиальных алгебр}. Эта теория, истоки которой лежат в конструкции алгебры Гриса, группой автоморфизмов которой является самая большая спорадическая простая группа монстр. Впервеы аксиальные алгебры были введены в работе \cite{HRS}. Это коммутативные неассоциативные алгебры порожденные идемпотентами особого вида, называемыми \textbf{осями}. Каждая ось порождает полупростой линейный оператор левого умножения, произведения совственных подпространств относительно этого оператора определяются своей таблицей слияния.
	
	%В случае когда таблица слияния $\mathbb{Z}_2$-градуирована с каждой осью связана $\mathbb{Z}_2$-градуировка алгебры и автоморфизм порядка $2$, называемым \textbf{автоморфизмом Миямото}. 

By the term "algebra" we mean a non-associative and non-commutative $\F$-algebra, where $\F$ is a field of characteristic not $2$. In this paper, we construct algebras $A_{\F}(G, T, \eta)$ associated with a group $GM(p)$-type $(G,T)$ and parameter $\eta\in \F$.

Algebras constructed from groups of $GM(p)$-type $(G,T)$ have dimension $|T|$ and a Miyamoto group is isomorphic to $G/Z(G)$, where $Z(G)$ is the center of $G$. The left multiplication operators by the generating idempotents have $4$ eigenvalues and the entire algebra is the sum of eigenspaces. In particular, the following theorem is proved.

\begin{Th}
Let $(G,T)$ be a $GM(p)$-type group and $A=A_{\F}(G,T,\eta)$, where $\eta\in \F$ be such that $1,\eta(p-2)+1, 1-\eta$ and $\eta-1$ are distinct elements of $\F$. The following assertions hold:
\begin{enumerate}
	\item All elements of $T$ are primitive left semisimple idempotents in $A$.

    \item If $\F$ contains the $\frac{p-1}{2}$ roots of $-1$, then any element $a\in T$ is a right semisimple idempotent in $A$.
	
    \item $A$ contains no right ideals.
	
	\item For each idempotent $a\in T$, the decomposition $A=A^+_a\oplus A^-_a$ defines a $\Z_2$-grading of $A$.	
	With respect to these gradings $Miy(A,T)\simeq G/Z(G)$.
	%\item if $T$ is finite set, then $Aut(A)$ is finite group.
\end{enumerate}
\end{Th}

%	Пусть $(G,T)$ -- группа строго $p$-транспозиций, $\mathbb F$-- поле характеристики $0$, $\mathbb F$-алгебра $A=A(G,T,\eta)$. Выполнены следующие утверждения:
	
%	1) алгебра $A$ простая;
	
%	2) если $G$ порождается подмножеством $\Omega\subseteq T$, то и $A$ как алгебра порождается множеством $\Omega$
	
%	3) Каждый элемент из $T$ является примитивным полупростым идемпотентом в $A$;
	
%	4) c каждым идемпотентом $a\in T$ связана $2$-градуировка $(A_a^+,A_a^-)$ алгебры $A$. Более того, если $b\in T$ и $a\neq b$, то $(A_a^+,A_a^-)$ и $(A_b^+,A_b^-)$ различные $2$-градуировки. 
	
%	5) группа $Miy(A,T)$ изоморфна $G/Z(G)$. 
%Очевидно, что если $T$ конечная группа, то и группа подстановок на множестве $T$ конечна. Следующий результат показывает, что в этом случае вся группа автоморфизмов алгебры $A(G,T,\eta)$ конечна.
%\begin{Th}
%	Пусть $(G,T)$ группа строго $n$-транспозиций. Если $T$ конечное множество, то $Aut(A(G,T,\eta))$ конечная группа.
%\end{Th}
We say that the field $\F$ has good characteristic relative to $(p,\eta)$ if the elements $1, \eta(p-2)+1, \eta-1$ and $1-\eta$ are distinct.

We also show that for $\eta=-\frac{1}{3}$ and a field $\F$ of good characteristic relative to $(5,\eta)$, the two generated algebra $A_{\F}(G,T,\eta)$ will be left axial with fusion law of Monster type.

\begin{prop}
	Let $(G,T)$ be a two generated group of $GM(5)$-type, $\eta=-\frac{1}{3}$ and $\F$ is a field of good characteristic relative to $(5,\eta)$. Every idempotent in $T$ is a left-axis of $\mathcal M(4/3,-4/3)$-type.
\end{prop}

The next two propositions allow us to distinguish between algebras for group $GM(5)$-type and groups $GM(p)$-type, where $p$ is a natural greater than $5$.

\begin{prop}
Let $(G,T)$ be a two generated group of $GM(p)$-type, where $p>5$, $\eta=-\frac{1}{p-2}$, $\F$ is a field of characteristic good related to $(p,\eta)$, and $A=A_{\F}(G,T,\eta)$. Every idempotent in $T$ is a left-axis of $\mathcal{GM}(\lm,-\lm)$-type.
\end{prop}
\begin{prop}
	Let $(G,T)$ be a group of $GM(p)$-type, where $p>5$, $\F$ is a field, $\eta\in \F$, and characteristic of $\F$ is good related to $(p,\eta)$ and $A=A_{\F}(G,T,\eta)$. All idempotents of $T$ are not a left $\mathcal M(\al,-\al)$-axis for any value of the parameter $\eta $.
\end{prop}

As a corollary of Theorems $2$ and $3$, we obtain a connection between the dimension of the algebra $A_{\F}(B(n,p).\la t\ra,T,\eta)$ and the order of the group $B(n,p)$.

\begin{cor}
Let $G= B(n,p).\la t\ra$, where $p$ is a prime greater than $2$ and $t$ is an automorphism of order $2$ of $B(n,p)$ that inverts free generators, $T=t^G$, $\F$ be a field of good characteristics relative to $(p,\eta)$. Then $Miy(A_{\F}(G,T,\eta),T)\simeq G/Z(G)$, where $Z(G)$ is the center of $G$ and $G$ is finite if and only if $A_ {\F}(G,T,\eta)$ is finite-dimensional.
\end{cor}

The obtained results allow us to reformulate Burnside's  conjecture on free Burnside groups of prime period $p$ generated by $d$ elements in terms of non-associative algebras.

\begin{question}
	For what $d$ and $p$ does the algebra $A_{\F}(B(d,p).\la t\ra, T, \eta)$, where $t$ and $T$ are as in Corollary 1, have finite dimension?
\end{question}

The definition of algebras $A_{\F}(G,T,\eta)$-type involves a group of $GM(p)$-type, and in fact we define the structure of an algebra on the set of conjugate involutions. However, such a definition can be avoided. We define an algebra $A=A(T,\eta)$ of $GM(\eta, p)$-type by defining multiplications on its basis $T$ and defining some properties of the basis $T$. It follows directly from the definition that for any group $(G,T)$ of $GM(p)$-type, the algebra $A_{\F}(G,T,\eta)$ is an algebra of $GM(\eta, p)$-type. But we will also prove the converse.

\begin{Th}
Let $A(T)$ be an algebra of $GM(p,\eta)$-type. Then $(M,\{\tau_t|t\in T\})=Miy(A,T)$ is a $GM(p)$-group and $A(T)\cong A_{\F}(M,\{\tau_t|t\in T\},\eta)$.
\end{Th}

Thus, Burnside's hypothesis on free Burnside groups of prime period $p$ generated by $d$ elements can be reformulated into a hypothesis that does not rely on the definition of a group.

\begin{question}
	For which $p$ does the algebra $GM(p,\eta)$-type generated by $d+1$ elements, is finite dimensional?
\end{question}

A large number of methods have now been developed for finding the dimension of an algebra. This reformulation of Burnside's conjecture opens up new approaches to its solution.

The paper is organized as follows. Section 1 presents preliminary results from group theory and proves Theorems 1 and 2. Section 2 provides basic definitions and preliminary results of non-associative algebras. In Section 3, Proposition 1, 2, and 3 are proved and useful results were obtained on $2$-generated algebras of $A_{\F}(G,T,\eta)$-type. In Section 4, Theorem 3 is proved. In Section 5 we define $GM(\eta,p)$-type algebras and prove Theorem 4.

\section{Elements of Group Theory}

Let $G$ be a group and $a,b\in G$ be involutions. We define the set $I(a,b)$ of all involutions of the dihedral group $\la a,b\ra$ generated by $a$ and $b$.

\begin{defi}
	A pair $(G,T)$ is called a faithful odd transposition group if $G$ is a periodic group and $T$ is a normal subset of involutions of $G$ such that $\la T\ra=G$, and $|ab|$ for any $a,b\in T$ is odd.
	%	Пара $(G,T)$ называется группой нечетных транспозиций если $G$ периодическая группа и $T$ нормальное подмножество инволюций из $G$ такое, что $\la T\ra=G$, и порядок произведения любых двух элементов из $T$ нечетен.	  
	\end{defi}
	
%	Введем важный подкласс групп нечетных транспозиций.
Let us introduce an important subclass of odd transposition groups.		
	\begin{defi}
		A faithful odd transposition group $(G,T)$ is called generalized Moufang type $n$, or $GM(n)$-type for short, if order of the product of any two distinct elements of $T$ is $n$.
		%Группа нечетных транспозиций $(G,T)$ называется обобщенно Муфангова типа $n$, коротко $GM(n)$-типа, если порядок произведения любых двух различных элементов из $T$ равен $n$.
	\end{defi}	

Note that if $n$ is not a prime and the product of involutions $a$ and $b$ has order $n$, then in the set $I(a,b)$ there exist two involutions $a'$ and $b'$ such that $|a'b'|<n$. Thus, if $(G,T)$ is a $GM(n)$-type group, then $n$ is a prime.
%Заметим, что если $n$ не простое нечетное число и произведение инволюций $a$ и $b$ имеет порядок $n$, то в множестве $I(a,b)$ найдутся две инволюции $a'$ и $b'$ такие, что $|a'b'|<n$. Таким образом, если $(G,T)$ группа $GM(n)$-type, то $n$-- простое число.   
	
	\begin{defi}
We say that $(G,T)$ a group of $GM(n)$-type is generated by a set $D$ if $D\subseteq T$ and $\la D\ra=G$.
%	Будим говорить, что $(G,T)$ группа строго $n$-транспозиций порождается множеством $D$ если $D\subseteq T$ и $D$ порождает группу $G$.  
	\end{defi}

Let $G$ be a finite group, $Z(G)$ be the center of $G$. We define the core of $G$, $K(G)$, to be the largest normal subgroup of odd order in $G$. Define $Z^*(G)$ to be the subgroup of $G$ containing $K(G)$ for which $Z^*(G)/K(G) = Z(G/K(G))$.
	
	\begin{lem}\cite[Corollary 1 (Glauberman’s $Z^{*}$-Theorem)]{Gl66}\label{Gl}
		Let $S$ be a Sylow $2$-subgroup of a finite group $G$. Suppose $x \in S$. A necessary and sufficient condition for $x \in Z^*(G)$ is that $C_G(x)\cap x^G=x$.
	\end{lem}
%Данное утверждение позволяет описать конечные группы нечетных транстпозиций.
The following statement allows us to describe finite odd transposition groups.

\bigskip
\noindent\textbf{Theorem 1.}
\textit{ Let $(G,T)$ be a finite faithful odd transposition group. Then $G\simeq K(G)\rtimes \la t\ra$, where $|t|=2$.
	%Пусть $(G,T)$ конечная группа нечетных транспозиций. Тогда $G\simeq K(G)\rtimes \la t\ra$.
}
\begin{proof}
Suppose that there exists a Sylow $2$-subgroup of $G$ containing two distinct commuting elements $a,b\in T$. But then $|ab|=2$; a contradiction with the fact that $(G,T)$ is a faithful odd transposition group. It follows from Lemma \ref{Gl} that $T\subseteq Z^*(G)$. Let $^-: G\rightarrow G/K(G)$ be the natural homomorphism. Since $\la T\ra=G$, it follows that $\la\overline{T}\ra=\overline{G}$. Note that $\overline{T}$ lies in $Z(\overline{G})$. Thus $\overline{G}$ is an elementary abelian $2$-group. Suppose that $|\overline{G}|>2$. Let $a,b\in T$ be such that $\overline{a}\neq\overline{b}$. We have $\overline{a}\overline{b}=\overline{ab}$, in particular $|\overline{ab}|=2$; a contradiction with the fact that $|ab|$ is odd and divisible by $|\overline{ab}|$.

%	Допустим, что найдется силовская $2$-подгруппа содержащая два различных коммутирующих элемента $a,b\in T$. Но тогда $|ab|=2$; противоречие с тем, что $(G,T)$ группа нечетных транспозиций. Из леммы \ref{Gl} следует, что $T\subseteq Z^*(G)$. Пусть $\overline{\ }: G\rightarrow G/K(G)$ естественный гомоморфизм. Из того, что $\la T\ra=G$ следует, что $\la\overline{T}\ra=\overline{G}$. Заметим, что $\overline{T}$ лежит в $Z(\overline{G})$. Таким образом $\overline{G}$ элементарная абелева $2$-группа. Допустим, что $|\overline{G}|>2$. Пусть $a,b\in T$ такие, что $\overline{a}\neq\overline{b}$. Имеем $\overline{a}\overline{b}=\overline{ab}$, в частности $|\overline{ab}|=2$; противоречие с тем, что $|ab|$ нечетен и делится на $|\overline{ab}|$.  
\end{proof}

\begin{lem}[Schreier’s Lemma  \cite{Sh26}]\label{sh26}
	A finite index subgroup of a finitely generated group is finitely generated.
	
\end{lem}
The following statement is a simple consequence of Schreier's Lemma.
%Следующее утверждение является простым следствием леммы Шрейера.

\bigskip
\noindent\textbf{Theorem 2.}
\textit{ 
	Let $(G,T)$ be an odd transposition group and let $G$ have a bounded period. If $G/Z(G)$ is a finite group and $T$ is a finite set, then $G$ is a finite group.
	%Пусть $(G,T)$ группа нечетных транспозиций и период группы $G$ ограничен. Если $G/Z(G)$ конечна и $T$ конечное множество, то и $G$ конечная группа.
}
\begin{proof}
Lemma \ref{sh26} implies that $Z(G)$ is a finitely generated subgroup. A finitely generated abelian group of finite exponent is finite. Thus $|G|=|G/Z(G)||Z(G)|<\infty$.
%	Из леммы \ref{sh26} следует, что $Z(G)$ конечно порожденная подгруппа. Конечно порожденная абелева группа конечного периода конечна. Таким образом $|G|=|G/Z(G)||Z(G)|<\infty$.  
\end{proof}

\section{Preliminaries in non-associative algebras}	
Let $\F$ be a field. For faithful odd transposition group $(G,T)$, we define a non-associative and non-commutative $\F$-algebra with basis $T$.	

\begin{defi}
Let $(G,T)$ be a faithful odd transposition group, $\Omega_T=\{|ab||a,b\in T\}\setminus\{1\}$. For each number $n\in\Omega_T$ we define an element $\eta_n\in\F\setminus\{0,1\}$. Let $\Gamma_T=\{\eta_n|n\in \Omega_T\}$.
We construct an $\F$-algebra with basis $T$ as follows:
$$
a * b = 
\begin{cases} 
a, & \text{if } b = a, \\
b^a+\delta(b^a), & \text{where } \delta(b^a)=\eta_n\sum_{x\in I(a,b)\setminus \{b^a\}}x \text{ and } n=|ab|.
\end{cases}
$$
We denote such an algebra by $A_{\F}(G,T,\Gamma_T)$ and call it the odd transposition algebra. In the case when $|\Omega_T|=1$ i.e. $(G,T)$ is a $GM(n)$-group, we write $A_{\F}(G,T,\eta)$.
%	Пусть $(G,T)$ группа нечетных транспозиций, $\Omega_T=\{|ab||a,b\in T\}$. Для каждого числа $n$ из $\Omega_T$ определим некоторый элемент $\eta_n\in\F$. Положим $\Gamma_T=\{\eta_n|n\in \Omega_T\}$. Построим $F$-алгебру с базисом $T$ следующим образом при $a\neq b$ положим $$a*b=b^a+\delta(b^a),\operatorname{where} \delta(b^a)=\eta_n\sum_{x\in I(a,b)\setminus b^a}x$$ и $$a*a=a$$. Обозначим такую алгебру $A(G,T,\Gamma_T)$ и будем называть ее алгеброй нечетных транспозиций, сокращенно $OT$-алгеброй. В случае когда $\Omega_H$ содержит одно число $n$, то будем писать $A(G,H,\eta)$. 
\end{defi}

%Элементы группы $G$ являются так же элементами алегебры $A(G,T,\Gamma_T)$. Мы зафиксируем, следующие обозначения.
The elements of the set $T$ are elements of the group $G$ and also elements of the algebra $A_{\F}(G,T,\Gamma_T)$. We will fix the following notation.
\begin{defi}
Let $(G,T)$ be a faithful odd transposition group and $A=A_{\F}(G,T,\Gamma_T)$, $\Omega\subseteq T$. We fix the following notation:
\begin{enumerate}
	\item $\la \Omega\ra$ is the subgroup of $G$ generated by the set $\Omega$;
	\item $\lla \Omega \rra$ is the subalgebra of $A_{\F}(G,T,\Gamma_T)$ generated by the set $\Omega$;
    \item when $|\Omega|=2$ we will call the algebra $\lla \Omega\rra$ a dihedral algebra;
    \item $L(\Omega)$ is the linear span of $\Omega$; 
    \item if $a,b\in T$, then $a^b=b^{-1}ab$ is the element of $T$ conjugated to $a$ via the element $b$;
    \item if $x=a_1+...+a_t$, where $a_i\in T, 1\leq i\leq t$, then $x^a=(a_1+...+a_t)^a=a_1^a+...+a_t^a$.
\end{enumerate}%Пусть $(G,T)$ группа нечетных транспозиций и $A=A(G,T,\Gamma_T)$, $\Omega\subseteq T$. Зафиксируем следующие обозначения.
%\begin{enumerate}
%	\item $\la \Omega\ra$ подгруппа группы $G$ порожденная множеством $\Omega;$
%	\item $\lla \Omega \rra$ подалгебра алгебры $A(G,T,\Gamma_T)$ порожденная множеством $\Omega$;
%	\item если $a,b\in T$, то $a^b=b^{-1}ab$ элемент из $T$ сопряженный с $a$ при помощи элемента $b$. Если $x=a_1+...+a_t$, где $a_i\in T, 1\leq i\leq t$, то $x=(a_1+...+a_t)_a=a_1^a+...+a_t^a$.
%\end{enumerate}	  
\end{defi}  

	%Пусть $(G,T)$ группа нечетных транспозиций. Заметим, что алгебра $A(G,T,\Gamma)$ коммутативна тогда и только тогда когда для любой пары $a,b\in T$ верно $a^b=b^a$. Если $|ab|>3$, то очевидно, что это условие не выполняется. Таким образом, $A(G,T,\Gamma)$ коммутативна только в случае когда $(G,T)$ группа строго $3$-транспозиций. 
	
	Let $(G,T)$ be the faithful odd transposition group. Note that the algebra $A_{\F}(G,T,\Gamma)$ is commutative if and only if for every pair $a,b\in T$ we have $a^b=b^a$. If $|ab|>3$, then this condition clearly fails. Thus, $A_{\F}(G,T,\Gamma)$ is commutative if and only if $(G,T)$ is $GM(3)$-type.
	
	\begin{defi} 
	Let $G=\langle a,b\rangle$ be the dihedral group generated by the involutions $a$ and $b$, of order $2p$, where $p$ is a prime. We enumerate the set $I(a,b)$ as follows: $a_1=a, a_2=b,..., a_{i+1}=a_{i-1}^{a_{i}}$. The basis $B(a,b)=\{a_1,...,a_p\}$ of the algebra $A_{\F}(G,I(a,b),\eta)$ will be called canonical.
	%Пусть $G=\langle a,b\rangle$ -- диэдральная группа порожденная инволюциями $a$ и $b$, порядка $2p$, где $p$ простое число. Перенумеруем множество $I(a,b)$ следующим образом: $a_1=a, a_2=b,..., a_{i+1}=a_{i-1}^{a_{i}}$. Базис $B(a,b)=\{a_1,...,a_p\}$ алгебры $A(G,I(a,b),\eta)$ будем называть каноническим.
	\end{defi}

	\begin{defi}
	A \emph{fusion law} $(\cF,\circ)$ over $\F$ is a finite set $\mathcal{F}$ of elements of $\F$ together
	with a map $\circ: \mathcal{F} \times \mathcal{F} \to 2^{\mathcal{F}}$. 
	\end{defi}
For small set $\mathcal{F}$, it is convenient to write the fusion law in table form. Figures \ref{M} and \ref{GM} show the fusion laws of the Monster type and generalized Monster type, respectively.
%Для небольших множеств $\mathcal{F}$ удобно закон слияния записывать в виде таблицы. На рисунках \ref{M} и \ref{GM} приведены таблицы слияния Монстрового и обобщенно монстрового типа соответственно.   

\begin{figure}[ht]
	\centering
	\begin{minipage}{0.45\textwidth}
		\centering
		\renewcommand{\arraystretch}{1.3}
		\begin{tabular}{|c||c|c|c|c|}
			\hline
			$\ast$ & $1$ & $0$ & $\alpha$ & $\beta$ \\
			\hline\hline
			$1$ & $1$ & $\varnothing$ & $\alpha$ & $\beta$ \\
			\hline
			$0$ & $\varnothing$ & $0$ & $\alpha$ & $\beta$ \\
			\hline
			$\alpha$ & $\alpha$ & $\alpha$ & $1,0$ & $\beta$ \\
			\hline
			$\beta$ & $\beta$ & $\beta$ & $\beta$ & $1,0,\alpha$ \\
			\hline
		\end{tabular}
		\caption{Monster type fusion law $\mathcal{M}(\alpha,\beta)$}\label{M}
		\label{fig:M1}
	\end{minipage}
	\hfill
	\begin{minipage}{0.45\textwidth}
		\centering
		\renewcommand{\arraystretch}{1.3}
		\begin{tabular}{|c||c|c|c|c|}
			\hline
			$\ast$ & $1$ & $0$ & $\alpha$ & $\beta$ \\
			\hline\hline
			$1$ & $1$ & $\varnothing$ & $\alpha$ & $\beta$ \\
			\hline
			$0$ & $\varnothing$ & $0,1$ & $\alpha$ & $\beta$ \\
			\hline
			$\alpha$ & $\alpha$ & $\alpha$ & $1,0,\alpha$ & $\beta$ \\
			\hline
			$\beta$ & $\beta$ & $\beta$ & $\beta$ & $1,0,\alpha$ \\
			\hline
		\end{tabular}
		\caption{Generalized Monster type fusion law $\mathcal{GM}(\alpha,\beta)$} \label{GM}
		\label{fig:GM1}
	\end{minipage}
\end{figure}

\begin{defi}
A fusion law $(\mathcal{F},\circ)$ is called $\Z_2$-graded if $\mathcal{F}=\mathcal{F}^+\cup\mathcal{{F}^-}$, $\mathcal{F}^+\cap\mathcal{F}^-=\emptyset$, and for any $\alpha^+,\beta^+\in \mathcal{F}^+$, and $\alpha^-,\beta^-\in \mathcal{F}^-$ the following statements hold $\alpha^-\circ\beta^-, \alpha^+\circ\beta^+\subseteq \mathcal{F}^+$, $\alpha^+\circ\beta^-,\alpha^-\circ\beta^+\subseteq\mathcal{F}^-$. In this case, $(\mathcal{F}^+,\mathcal{F}^-)$ is called a $\Z_2$-grading of the fusion law $\mathcal{F}$. If the set $\mathcal{F}^-$ is empty, then the $\Z_2$-grading $(\mathcal{F}^+,\mathcal{F}^-)$ is called trivial. 
%В этом случае $(\mathcal{F}_+,\mathcal{F}_-)$ называется 2-градуировкой закона слияния $(\mathcal{F})$. Если множество $\mathcal{F}_-$ пустое, то $2$-гадуеровка называется тривиальной.  
\end{defi}

Note that the fusion laws $\cal{M}(\al,\bt)$ and $\cal{GM}(\al,\bt)$ are $\Z_2$-graded. The decomposition $(\{0, 1, \al\}, \{\bt\})$ is a $\Z_2$-grading in both cases.

\begin{defi}
    An $\F$-algebra $A$ is called $\Z_2$-graded if there exist subspaces $A^+$ and $A^-$ such that $A=A^+\oplus A^-$, $A^+\cap A^-=0$, and for any $\alpha^+,\beta^+\in A^+$, and $\alpha^-,\beta^-\in A^-$ the following statements hold $\alpha^-*\beta^-, \alpha^+*\beta^+\in A^+$, $\alpha^+*\beta^-,\alpha^-*\beta^+\in A^-$. In this case, $(A^+,A^-)$ is called a $\Z_2$-grading of $A$. If the subspace $A^-$ is empty, then the $\Z_2$-grading $(A^+,A^-)$ is called trivial. 
\end{defi}
%Пусть $A$ неассоциативная и не коммутативная $\F$-алгебра.
Let $A$ be a non-associative and non-commutative $\F$-algebra.
Put $a \in A$ be an idempotent, $X\in \{L, R\}$, and $\lambda, \delta \in \F$.
We define the left and right multiplication maps $L_a(b):=a*b$ and	$R_a(b):=b* a$. Write $A_\lambda(X_a) = \{v \in A : X_a(v) = \lambda v\}$. Note that it may happen that  
	$A_\lambda(X_a) = 0$. If $\Omega$ is a set of elements of $\F$, then $A_{\Omega}(X_a)=\bigoplus_{\alpha\in\Omega}A_{\alpha}(X_a)$.
	
	Since $a$ is an idempotent, then $dim(A_1(X_a))\geq 1$.
	We say that an idempotent $a$ is primitive if $dim(A_1(L_a))=dim(A_1(R_a))=1$. A linear operator $L$ on an algebra $A$ is called semisimple if $A$ is decomposed into a sum of eigenspaces of $L$. An idempotent $a$ is called left (right) semisimple if the linear operator $L_a$ ($R_a$) is semisimple. 

\begin{defi}

Let $(\cF,\circ)$ be some fusion law, $a\in A$ be a primitive semisimple idempotent such that $A=A_{\cF}(X_a)$, and for any $\alpha,\beta \in \cF$ and any $x\in A_{\alpha}(X_a), y\in A_{\beta}(X_a)$, we have $x* y\in A_{\alpha\circ\beta}(X_a)$. We say that $a$ is the left $(\mathcal{F},\circ)$-axis if $X=L$ and $a$ is the right $(\mathcal{F},\circ)$-axis if $X=R$.

%Пусть $a\in A$ примитивный полупростой идемпотент такой, что $A=A_{\mathcal{F}}(X_a)A$, и для любых $\alpha,\beta \in \mathcal{F}$ и любых $x\in A_{\alpha}(X_a), y\in A_{\beta}(X_a)$ верно $x\cdot y\in A_{\alpha*\beta}(X_a)$, где $X\in\{L,R\}$. Будем говорить, что $a$ левая $(\mathcal{F},*)$-ось если $X=L$ и $a$ правая $(\mathcal{F},*)$-ось если $X=R$. 	
\end{defi}	

In particular, if for some idempotent $a$ the operator $L_a$ is semisimple, then $a$ is a left axis with the maximally general fusion law.
    
If the algebra is commutative, then $R_a=L_a$ and therefore, instead of the right $(\mathcal{F},\circ)$-axis we can say $(\mathcal{F},\circ)$-axis.

Let $(\F,\circ)$ be a $\Z_2$-graded fusion law, and $a\in A$ be the left (right) $(\F,\circ)$-axis. Then $Gr(L_a)=(A_{\cF^+}(L_a),A_{\cF^-}(L_a))$ ($Gr(R_a)=(A_{\cF^+}(R_a),A_{\cF^-}(R_a))$) is a $\Z_2$-grading of $A$.

\begin{defi}	
An algebra $A$ is called left (right) $(\cF,\circ)$-axial if it is generated by left (right) $(\cF,\circ)$-axes.
%Алгебра $A$ называется лево(право) $(\mathcal{F},*)$-аксиальной если порождена левыми(правыми) $(\mathcal{F},*)$-осями.
\end{defi}

Let $(\cF^+,\cF^-)$ be a $\Z_2$-grading of the fusion law $(\cF,\circ)$, and $A$ be a left or right $(\cF,\circ)$-axial algebra. In this case, for any $\cF$-axis $a\in A$, we have $(A_{\cF^+}(L_a), A_{\cF^-}(L_a))$ is a $\Z_2$-grading of $A$.

%Пусть $(\mathcal{F}_+,\mathcal{F}_-)$ это 2-градуировка закона слияния $(\mathcal{F},*)$, и $A$ левая или правая $(\mathcal{F},*)$-аксиальная алгебра порожденная множеством осей $\Omega$. В этом случае для любой оси $a\in\Omega$ разложение алгебры $A$ в сумму $A_{\mathcal{F}_+}(X_a)\oplus A_{\mathcal{F}_-}(X_a)$ является 2-градуировкой алгебры $A$, где $X\in\{L,R\}\}$.

\begin{defi}
Let $Gr=(A^+,A^-)$ be a $\Z_2$-grading of $A$. We define the mapping $\tau(Gr):A\rightarrow A$ as follows $\tau(Gr)x^+=x^+, \tau(Gr)x^-=-x^-$, where $x^+\in A^{+}$ and $x^-\in A^{-}$. 
%Пусть $a\in \Omega$. Определим отображение $\tau_a:A\rightarrow A$ следующим образом $\tau_ax_+=x_+, \tau_ax_-=-x_-, x_+\in A_{\mathcal{F}_+}(X_a), x_-\in A_{\mathcal{F}_-}(X_a)$. Отображение $\tau_a$ является изоморфизмом алгебры $S$ и называется $\tau$-инволюцией оси $a$.  
\end{defi}

The map $\tau(Gr)$ is an automorphism of the algebra $A$.
If a $\Z_2$-grading $Gr$ is defined as a $\Z_2$-grading of some $(\mathcal{F},\circ)$-axis $a$, then we will write $\tau_a$ instead $\tau(Gr)$, and say that $\tau_a$ is a Miyamoto involution of $a$.
\begin{lem}\cite[Lemma 5.1]{HRS2}\label{tauaut}
If $t$ is an automorphism of $A$ and $a$ is an axis, then $a^t$ is an axis with
$\tau(a^t) = \tau(a)^t$.
\end{lem}

\begin{defi} Let $\Omega$ be a set of $(\cF,\circ)$-axes of algebra $A$. 
The group $Miy(A,\Omega)=\la \tau_a|a\in \Omega\ra$ is called a Miyamoto group of the set $\Omega$.
%	Группа $Miy(\Omega)=\la \tau_a|a\in \Omega\ra$ называется группой Миямото множества $\Omega$.
\end{defi}

	%\item We write $ A_{\lambda,\delta}(a) := A_\lambda(L_a) \cap A_\delta(R_a).$ We just write $A_{\lambda,\delta}$ when the idempotent $ a$ is understood. An element in $A_{\lambda,\delta}(a) $ will be called a $ (\lambda, \delta) $-eigenvector of $a $, and $ (\lambda, \delta) $ will be called its eigenvalue.
  
\begin{defi}	
	A bilinear form $(\cdot ,\cdot )$ on an algebra $A$ is called a left (right) Frobenius form if for any $a,b,c\in A$ it is true $(ab,c)=(b,ac)$ ($(ab,c)=(a,cb)$).
	%Билинейная форма $(\ ,\ )$ на алгебре $A$ называется левой(правой) формой фробениуса если для любых $a,b,c\in A$ верно $(ab,c)=(b,ac)$ ($(ab,c)=(a,cb)$).	
\end{defi}	

\section{2-generated $A_{\F}(G,H,\eta)$ algebra}
%In this paragraph, $\F$ is a field of a characteristic other than $2$.
Let $G=\langle a,b\rangle$ be the dihedral group of order $2p$, where $p$ is a prime, generated by involutions $a$ and $b$, $T=I(a,b)$, $\eta\in\F\setminus\{0,1\}$, $A=A_{\F}(G,T,\eta)$, $B=B(a,b)$.	
%Пусть $G=\langle a,b\rangle$ -- диэдральная группа порожденная инволюциями $a$ и $b$, порядка $2p$, где $p$ простое число, $T=a^G$, и $\eta_p\in\F$, $A=A(G,T,\eta_p)$, $B=B_G(a,b)$. 
%Основная цель данного параграфа доказать, что с каждый идемпотент из $T$ является $\mathcal{D}$ осью и $M(\{a,b\})$ изоморфную $G/Z(G)$.  

%Запишем матрицу $L_{a}$ в каноническом базисе $B_G(a,b)$.  
Let us write the matrix $L_{a}$ in the canonical basis $B(a,b)$.
	$$
	Mat(L_{a})=\begin{pmatrix}
		1    & 0    &\cdots&0    &0 \\
		\eta & \eta &\cdots& \eta& 1\\ 
		\eta & \eta &\cdots& 1& \eta\\
		\vdots&\vdots& \ddots & \vdots&\vdots\\
		\eta & 1&\cdots& \eta& \eta\\
	\end{pmatrix}
	$$
	
	%Найдем спектр матрицы $Mat(L_{a}(A))$. 
	We find the spectrum of the matrix $Mat(L_{a})$.
	\begin{lem}\label{sv}
%	Собственными числами матрицы $Mat(L_a)$ являются:
The eigenvalues of the matrix $Mat(L_a)$ are:
	\begin{enumerate}
		\item $1$ is an eigenvalue of multiplicity $1$;
		\item $\lambda_1 = \eta(p-2) + 1$ is an eigenvalue of multiplicity $1$;
		\item $\lambda_2 = 1-\eta$ is an eigenvalue of multiplicity $\frac{p-3}{2}$;
		\item $-\lambda_2=\eta-1$ is an eigenvalue of multiplicity $\frac{p-1}{2}$. 
	\end{enumerate}
		
	%$Spec(Mat(L_{a}(A)))=\{1, 1 + (p-2)\eta_p, 1 - \eta_p$ собствееное число кратности $(p-1)/2-1$, $-(1 - \eta_p)$ собственное число кратности $(p-1)/2\}$. 
	\end{lem}
\begin{proof}
%Утверждение леммы следует, из того, что матрица $Mat(L_a)$ имеет $p$ собственных векторов см. Лемму \ref{basis2} с собственными числами $1,\lambda_1,\lambda_2,-\lambda_2$ нужной кратности.
The claim follows from the existence of $p$ eigenvectors of the matrix $\mathrm{Mat}(L_a)$ (see Lemma~\ref{basis2}) corresponding to the eigenvalues $1$, $\lambda_1$, $\lambda_2$, and $-\lambda_2$ with the stated multiplicities.
\end{proof}

Let $\lm_1$ and $\lm_2$ be elements of $\F$ as in Lemma \ref{sv}.
%If the elements $\lm_1, \lm_2$ and $-\lm_2$ are different from $1$, then $a$ is a primitive idempotent. 
We now find the eigenspaces of the matrix $Mat(L_{a})$.
\begin{lem}\label{basis2}
	%Алгебра $A$ имеет базис составленный из собственных векторов матрицы $Mat(L_a)$. Запишем собственные вектора в базисе $B_G(a,b)$:
	The algebra $A$ has a basis of eigenvectors of the matrix $Mat(L_a)$. Let's write the eigenvectors in the basis $B(a,b)$:
	\begin{enumerate}
		\item $a = (1,0,\dots,0)$ is an eigenvector corresponding to the eigenvalue $1$;
		
		\item $z(a) = \left(\frac{p-1}{p-2}, 1, \dots, 1\right)$ is an eigenvector corresponding to the eigenvalue $\lambda_1$.
		
		\item  
		\begin{align*}
			y_1(a) &= (0, 1, 0, \dots, 0, -1, -1, 0, \dots, 0, 1), \\
			y_2(a) &= (0, 0, 1, 0, \dots,0, -1, -1, 0, \dots,0, 1, 0), \\
			&\vdots \\
			y_{\frac{p-3}{2}}(a) &= (0, \dots,0, 1, -1, -1, 1, 0, \dots, 0),
		\end{align*}
		are eigenvectors corresponding to the eigenvalue $\lambda_2$.
		
		\item \begin{align*}
			x_1(a) &= (0, -1,0, \dots,0, 1), \\
			x_2(a) &= (0, 0, -1,0, \dots,0, 1, 0), \\
			&\vdots \\
			x_{\frac{p-1}{2}}(a) &= (0, \dots, 0, -1, 1, 0, \dots, 0),
		\end{align*}
		are eigenvectors corresponding to the eigenvalue $-\lambda_2$.
	\end{enumerate}
\end{lem}

\begin{proof}
	It is easy to verify that this set of vectors are indeed eigenvectors of the matrix $Mat(L_a)$. It is also straightforward to prove that this set of vectors is linearly independent, and therefore, it forms a basis for the algebra $A$.
%Легко проверить, что данное множество векторов действительно являются собственными векторами матрицы $Mat(L_a)$. Так же несложно доказывается, что данное множество векторов является линейно независимым, а значит, исходя из его порядка, является базисом алгебры $A$. 
\end{proof}

%Таким образом $a$ является полупростым идемпотентом алгебры $A$.
%Обозначим %$BE(L_a,\{a,b\})=\{a,z(a,b),y_1(a,b),...,y_{(p-1)/2-1}(a,b),x_1(a,b),...,x_{(p-1)/2}(a,b)\}$ базис из собственных векторов линейного оператора $L_a$. Обозначим через $A^+_a$ пространство натянутое на вектора $a,z,y_1(a,b),...,y_{(p-3)/2}(a,b)$, и $A^-_a$ пространство натянутое на $x_1(a,b),...,x_{(p-1)/2}(a,b)$. 
%Докажем, что $A^+_a$ является подалгеброй, произведение любых двух элементов из $A^-_a$ лежит в $A^+_a$, а произведение элемента из $A^+_a$ и элемента из $A^-_a$ лежит в $A^-_a$.  

%Запишем таблицу умножения алгебры $A$ в базисе $BE(L_a,\{a,b\})$.
%Множество $BE(L_a,\{a,b\})$ является множеством собственных векторов относительно умножения на $a$ слева, поэтому умножения на $a$ слева уже известны.
%Вычислим произведения между векторами базиса подпространства $A_a^+$.

Thus, $L_a$ is a semisimple linear operator of the algebra $A$ and therefore $a$ is a left semisimple element.

Let $BE(L_a,{a,b}) = \{a, z(a), y_1(a),\dots,y_{(p-3)/2}(a), x_1(a),\dots,x_{(p-1)/2}(a)\}$ be the basis of $A$ consisting of eigenvectors of the linear operator $L_a$. 
\begin{defi}
For an element $a\in T$ we define 
   
$$A^+_a=\{a(a), z(a), y_1(a),\dots,y_{(p-3)/2}(a)\}$$ and $$A^-_a=\{x_1(a),\dots,x_{(p-1)/2}(a)\}$$    
\end{defi}

We prove that $A^+_a$ is a subalgebra, the product of any two elements from $A^-_a$ lies in $A^+_a$, and the product of an element from $A^+_a$ with an element from $A^-_a$ lies in $A^-_a$.

Let us write the multiplication table of $A$ in the basis $BE(L_a,{a,b})$. The set $BE(L_a,{a,b})$ consists of eigenvectors with respect to left multiplication by $a$, so the results of left multiplication by $a$ are already known.

We will now compute the products between the vectors of the basis of the subspace $A_a^+$.

\begin{lem}\label{bemgam}
	Let $\gam=\sum_{a_j\in B(a,b)}a_j$. The following statements hold. 
	\begin{enumerate}
	\item $a_i*\gam=\gam*a_i=(\eta(p-2)+1)\gam+\eta a_i$;
	\item $\gam*\gam=r\gam$, where $r=\eta(p-1)^2+p$. 
\end{enumerate}
\end{lem}
\begin{proof}
	From Lemma \ref{basis2} it follows that $\gam+\frac{1}{p-2}a_i$ is an eigenvector of the operator $L_{a_i}$ with eigenvalue $\eta(p-2)+1$. Thus, 
    $$a_i*\gam=(\eta(p-2)+1)\Big(\gam+\frac{1}{p-2}a_i\Big)-\frac{1}{p-2}a_i=(\eta(p-2)+1)\gam+\eta a_i.$$
    
	Since $\gam$ does not change under permutation of the elements of the basis $B(a,b)$, it suffices to verify that $\gam*a_1=a_1*\gam$. We have $$\gam*a_1= a_1+\eta(\gam-a_1^{a_2})+a_1^{a_2}+...+\eta(\gam-a_1^{a_p}) +a_1^{a_p}= $$
	$$=(\eta(p-2)+1)\gam+\eta a_1.$$
	2. Based on point 1, we have $$\gam*\gam=p(\eta(p-2)+1)\gam+\eta a_1+...+\eta a_p=(p(\eta(p-2)+1)+\eta)\gam=(\eta(p-1)^2+p)\gam.$$
	%1. Из Леммы \ref{basis2} следует, что $\gam+\frac{1}{p-2}a_i$ является собственным вектором для оператора $L_{a_i}$ с собственным числом $\eta(p-2)+1$. Таким образом $$a_i*\gam=(\eta(p-2)+1)(\gam+\frac{1}{p-2}a_i)-\frac{1}{p-2}a_i=(\eta(p-2)+1)\gam+\eta a_i.$$
	%Поскольку $\gam$ не изменится при перестановки местами элементов базиса $B(a,b)$, то достаточно проверить, что $\gam*a_1=a_1*\gam$. Имеем $$\gam*a_1= a_1+\eta(\gam-a_1^{a_2})+a_1^{a_2}+...+\eta(\gam-a_1^{a_p}) +a_1^{a_p}= $$
	%$$=(\eta(p-2)+1)\gam+\eta a_1.$$ 
	%2. Исходя из пункта 1. имеем $$\gam*\gam=p(\eta(p-2)+1)\gam+\eta a_1+...\eta a_p=(p(\eta(p-2)+1)+\eta)\gam.$$  
\end{proof}
\begin{rem}
Note that for $\eta=-\frac{1}{p-2}$ the element $(2-p)\gam$ is the unit of the algebra $A$.
%	Заметим, что при $\eta=-\frac{1}{p-2}$ элемент $(2-p)\gam$ является единицей алгебры $A$.
\end{rem}

\begin{lem}\label{prod1}
Products of basis elements of the space $A_a^+$:
%Произведения базисных элементов пространства $A_a^+$.
\begin{enumerate}
\item let $p>5$, If $i\leq \frac{p-3}{4}$ then $y_i(a)*a=(1-\eta)(y_{2i}(a)-y_1(a))$ else $y_i(a)*a=(1-\eta)(y_{p-3-2i}(a)-y_1(a))$;

\item let $p=5$, we have $y_1(a)*a=(\eta-1)y_1(a)$;

\item $z(a)*y_i(a)= \frac{\eta(p-3)+1}{p-2}y_i(a)$;

\item $y_i(a)*z(a)=y_i(a)+y_i(a)*a\in L(y_1,...,y_{(p-3)/2})$;

\item $z(a)*z(a)=\frac{1+2\eta(p-2)}{(p-2)^2}a+(\eta(p-1)^2+p)\gam\in L(a,z)$; 

\item $z(a)*a=a*z(a)=(1+(p-2)\eta)z(a)$;

\item $y_i(a)*y_i(a), y_i(a)*y_j(a)\in A_a^+$;

\item if $p=5$, then $y_1(a)*y_1(a)=(1-4\eta)z(a)-\frac{16}{3}(1-\eta)a.$ 
\end{enumerate}
\end{lem}
\begin{proof}
	
1. We have $$ y_i(a)*a=(a_{i+1}-a_{\frac{p+1}{2}}-a_{\frac{p+3}{2}}+a_{p+1-i})*a=$$

$$=a_{2i+1}+\delta(a_{2i+1})-a_p-\delta(a_p)-a_2-\delta(a_2)+a_{p+1-2i}+\delta(a_{p+1-2i})= $$
$$=(1-\eta_p)(a_{2i+1}-a_p-a_2+a_{p+1-2i})=(1-\eta_p)(y_{2i}(a)-y_1(a))$$
2. It is proved similarly as the case 1.\\
3. Let $\gam$ be as in Lemma \ref{bemgam}. Note that $z(a)=\frac{1}{p-2}a+\gam$. From Lemma \ref{bemgam} we have $\gam*a_l=(\eta(p-2)+1)\gam+\eta a_l$. Hence $\gam*y_i(a)=\eta y_i(a)$.Thus $$(\frac{1}{p-2}a+\gam)*y_i(a)=\frac{\eta(p-3)+1}{p-2}y_i(a)$$
4. We have $$ y_i(a)*(\gam +\frac{1}{p-2}a)=y_i(a)+y_i(a)*a$$ \\
5. From Lemma \ref{bemgam} we have $\gam*\gam=(\eta(p-1)^2+p) \gam$. Therefore $$z(a)*z(a)=(\frac{1}{p-2}a+\gam)(\frac{1}{p-2}a+\gam)=\frac{1}{(p-2)^2}a+\frac{1}{p-2}(\gam*a+a*\gam)+\gam*\gam=$$
$$=\frac{1+2\eta(p-2)}{(p-2)^2}a+(\eta(p-1)^2+p)\gam.$$
\\
6. We have $$z(a)*a=(\frac{1}{p-2}a+\gam)*a=(\frac{1}{p-2}+\eta)a+(\eta(p-2)+1)\gam=$$
$$=(\frac{1-(\eta(p-2)+1)}{p-2}+\eta)a+ (\eta(p-2)+1)z(a)=(\eta(p-2)+1)z(a)$$\\
7. and 8. Direct calculations.
\end{proof}
\begin{rem}\label{Raeigenvector}
 It follows from Lemma \ref{prod1} that $z(a)$ is an eigenvector of $R_a$ with eigenvalue $1+(p-2)\eta$,  $y_i(a)$ are eigenvectors with eigenvalues $-\lm_2$. Thus $R_a$ is semisimple on $A^{+}(a)$.
 
 %  Из леммы \ref{pro1} следует, что $z(a)$ собственный вектор $R_a$ с собственным значением $1+(p-2)\eta$, Элементы $y_i(a)$ являются собственными векторами с собственными значениями $\eta-1$. Таким образом $R_a$ полупростой на $A^{+}(a)$.
\end{rem}

\begin{lem}\label{aplus}
	The subspace $A_a^+$ is a subalgebra of $A$. 
\end{lem}
\begin{proof}
	By definition, $Bas(A_a^+)=\{a, z, y_1,..., y_{(p-3)/2}\}$ is a basis of $A_a^+$. From Lemma \ref{prod1} it follows that the product of any two vectors of $Bas(A_a^+)$ is linearly combination of vectors of $Bas(A_a^+)$. Thus, $\lla Bas(A_a^+)\rra=A_a^+$.
%По определению $Bas(A_a^+)=\{a,z,\alpha_1,...,\alpha_{(p-3)/2}\}$ является базисом пространства $A_a^+$. Из леммы \ref{prod1} следует, что произведение любых двух векторов из $Bas(A_a^+)$ линейно выражается через вектора из $Bas(A_a^+)$. Таким образом $\lla Bas(A_a^+)\rra=A_a^+$. 
\end{proof}

%Докажем, что $A_a^-$ выдерживает умножение на элементы из $A_a^+$.
We show that $A_a^-$ is invariant under multiplication by elements of $A_a^+$.

\begin{lem}\label{prod2}
	
\begin{enumerate}
\item $x_i*a=(-1)^k(1-\eta)x_{l}$, where $k$ is the integer part of $2i$ divided by $\frac{p-1}{2}$ and $l$ is the remainder of $2i$ divided by $\frac{p-1}{2}$;  

\item $z*x_i=(\frac{\eta p-\eta-1}{p-2})x_i$;

\item If $i\leq\frac{p-1}{4}$, then $x_i*z= \eta x_i+ \frac{\eta-1}{p-2}x_{2i}$, else $x_i*z= \eta x_i- \frac{\eta-1}{p-2}x_{p-2i}$;

\item $y_j*x_i\in A_a^{-}$;

\item $x_i*y_j\in A_a^{-}$.
\end{enumerate}
\end{lem}
\begin{proof}
Points 1, 2, and 5 are easy to calculate.\\
Points 3, 4 are easily proven using the fact that $z=\frac{1}{p-2}a+\gam$ and Lemma \ref{bemgam}.
\end{proof}
As a corollary of Lemmas \ref{prod1} and \ref{prod2} we obtain that $R_a$ acts on eigenspaces of $L_a$.
\begin{cor}\label{ra}
The following statements are true:
 \begin{enumerate}
		\item $A_{\lm_1}(L_a)=A_{\lm_1}(R_a)$;
		\item $R_a(A_{\lm_2}(L_a))=A_{\lm_2}(L_a)$;
		\item $R_a(A_{-\lm_2}(L_a))=A_{-\lm_2}(L_a)$.
\end{enumerate}
\end{cor}
\smallskip

Simple calculations lead us to the fact that the product of any two elements from $A^-_a$ lies in $A^+_a$.
\begin{lem}\label{prod3}
$x_i*x_j\in A_a^+$.	
\end{lem}

\begin{lem}\label{ra2}
    The linear operator $R_a$ is primitive. If $\F$ contains $\frac{p-1}{2}$-th roots of $-1$, then $R_a$ is semisimple.
\end{lem}
\begin{proof}
It follows from Lemma \ref{prod1} that $z, y_1, \dots, y_{\frac{p-3}{2}}$ are eigenvectors of $R_a$. Since $A_a^-$ is an invariant subspace with respect to $R_a$, it suffices to show that $A_a^-$ decomposes into a sum of eigenspaces of $R_a$ and that none of the eigenvalues of $R_a$ acting on $A_a^-$ is $1$. 

According to Lemma \ref{prod2}, the matrix of $R_a$ acting on $A_a^-$, written in the basis $x_1, \dots, x_{\frac{p-1}{2}}$, is a monomial matrix with non-zero entries $1-h$ and $-(1-h)$. Thus, the roots of the characteristic polynomial are $d(h-1)$, where $d$ is a $\frac{p-1}{2}$-th root of $-1$. Consequently, $R_a$ is a primitive operator. If $\F$ contains the $\frac{p-1}{2}$-th roots of $-1$, then $R_a$ is semisimple on $A$.

    %Из леммы \ref{prod1} следуут, что $z_a, y_1,..., y_{\frac{p-3}{2}}$ являются собственными векторами $R_a$. Достаточно показать, что $A_a^-$ разлогается в сумму собственных подпространств $R_a$ и среди собственных чисел $R_a$ при действии на $A_a^-$ нет $1$.
    %Из леммы \ref{prod2} следует, что матрица $R_a$ при действии на $A_a^-$ записанная в базисе $x_1,..., x_{\frac{p-1}{2}}$ является мономиальной с ненуливыми значениями $1-h$ и $-(1-h)$. Таким образом корнями характеристического многочлена являются $d(h-1)$, где $d$ это корень из $-1$ степени $\frac{p-1}{2}$. Следовательно, $R_a$ примитивный оператор. Если $\F$ содержит корни из $-1$ степени $\frac{p-1}{2}$, то $R_a$ полупростой на $A$. 
\end{proof}

%\begin{lem}
%We define a linear transformation $d_a$ on $A$ as follows $d_a(x)=x*a+a*x$. If there exists $y\in A$ such that $d_a(y)=y+\sigma a$, then $y\in L(a)$.
	%Определим линейное преобразование $d_a$ на $A$ следующим образом $d_a(x)=x*a+a*x$. Если найдется $d_a(y)=y+\sigma a$, то $y\in L(a)$.
%\end{lem}	
%\begin{proof}
%	Имеем $y=\al a+\bt z+ \gam_1 y_1+..+\sigma x_1+...$. Из Следствия \ref{ra} следует, что $\sigma=\al=0$. Предположим, что $\bt\neq0$. Из Леммы \ref{prod1} следует, что $d_a(y)_{A_{\lm_1}(a)}=2\lm_1\bt z$. Таким образом если $\bt\neq 0$, то $\lm_1=\frac{1}{2}$, в частности $\eta=-\frac{1}{2(p-2)}$    
%\end{proof}

%\begin{table}[h]
%	\centering
%	\begin{tabular}{|c||c|c|c|c|}
%		\hline
%		$\ast$ & $1$ & $\lm_1$ & $\lm_2$ & $-\lm_2$ \\
%		\hline\hline
%		$1$ & $1$ & $\lm_1$ & $\lm_2$ & $-\lm_2$ \\
%		\hline
%		$\lm_1$ & $\lm_1$ & $\{1,\lm_1\}$ & $\lm_2$ & $-\lm_2$ \\
%		\hline
%		$\lm_2$ & $\lm_2$ & $\lm_2$ & $\{1,\lm_1,\lm_2\}$ &  $-\lm_2$ \\
%		\hline
%		$-\lm_2$ & $-\lm_2$ & $-\lm_2$ & $-\lm_2$ & $\{1,\lm_1,\lm_2\}$ \\
%		\hline
%	\end{tabular}
%	\caption{Fusion law $\mathcal{G}(\lm_1,\lm_2,-\lm_2)$-type} \label{G}
%\end{table} 

The following 2 Propositions are obtained as corollaries of Lemmas \ref{prod1}, \ref{prod2}, \ref{prod3}.

\bigskip
\noindent
\textbf{Proposition 1.} \textit{	Let $(G,T)$ be a two generated group of $GM(5)$-type, $\eta=-\frac{1}{3}$ and $\F$ is a field of good characteristic relative to $(p,\eta)$. Every idempotent in $T$ is a left-axis of $\mathcal M(4/3,-4/3)$-type.}

\bigskip
\noindent\textbf{Proposition 2.} \textit{	Let $(G,T)$ be a two generated group of $GM(p)$-type, where $p>5$, $\eta=-\frac{1}{p-2}$, $A=A_{\F}(G,T,\eta)$ and $\F$ is a field of characteristic good related to $(p,\eta)$. Every idempotent in $T$ is a left-axis of $\mathcal{GM}(\lm,-\lm)$.
}

\bigskip
\noindent\textbf{Proposition 3.} \textit{Let $(H,T')$ be a group of $GM(p)$-type, where $p>5$, $A'=A(H,T',\eta)$. Any idempotent in $T'$ is not a primitive left $\mathcal M(\al,-\al)$-axis for any value of the parameter $\eta $.}

\begin{proof}
	Suppose that $a\in T'$ is a primitive left $\mathcal M(\al,-\al)$-axis of $A'$. Let $b\in T\setminus\{a\}$. 
	Note that $\lm_2$ cannot be equal to $0$. Consequently, $\lm_1=0$ and, therefore $\eta=-\frac{1}{p-2}$. It follows from Lemma \ref{prod1} that the projection of $y_i*z$ onto $L(a)$ is non trivial for $p>5$. Thus, $a$ is not a primitive left $\mathcal M(\al,-\al)$-axis of $A'$.
\end{proof}	

\begin{cor}\label{twogenaxial}
	%Разложение алгебры $A$ на $A^+_a$ и $A^-_a$ является $2$-градуировкой алгебры $A$, более того верны следующие утверждения:
	For each idempotent $c\in T$, the decomposition of the algebra $A$ into a direct product of the subspaces $A^+_c$ and $A^-_c$ is a $\Z_2$-grading of the algebra $A$.
\end{cor}

Consequently, an automorphism $\tau_c $ of $A$ is defined.

\begin{lem}\label{xxa}
For any $x\in A$, $x+x^a\in A^+_a$.
%Для любого $x\in A$ выполнено $x+x^a\in A^+_a$.
\end{lem}
\begin{proof}
Since $x$ is the sum of elements from $T$, it is sufficient to prove that $b+b^a\in A^+_a$. We have:	$$\frac{b+b^a}{2}=\frac{z-\frac{p-1}{p-2}a-\sum_{2}^{(p-1)/2-1} y_i +((p-1)/2-1)y_1}{(p-3)}$$  
\end{proof}
\begin{lem}\label{xmxa}
	For any $x\in A$, $x-x^a\in A^-_a$ holds.
	%Для любого $x\in A$ выполнено $x+x^a\in A^+_a$.
\end{lem}
\begin{proof}
	Since $x$ is the sum of elements from $T$, it is sufficient to prove that $b-b^a\in A^-_a$. If $b=a$, then $b-b^a=0$, else $b-b^a=x_1$ and therefore $b-b^a\in A_a^-$.     
\end{proof}

Let us construct a map $f: \langle \tau_a, \tau_b \rangle \to G$ such that $f(\tau_a)=a$ and $f(\tau_b)=b$.

\begin{lem}\label{Mi2}
	For any $c,d\in T$, we have $\tau_c(d)=d^c$.
\end{lem}
\begin{proof}
	Let $d^+=\frac{1}{2}(d+d^c)$ and $d^-=\frac{1}{2}(d-d^c)$. We have $d=d^++d^-$. From Lemmas \ref{xxa} and \ref{xmxa} it follows that $d^+\in A^+_c$ and $d^-\in A^-_c$. Thus $\tau_c(d)=\tau_c(d^++d^-)=d^+-d^-=d^c$. 
\end{proof}

\begin{lem}
	The mapping $f$ is an isomorphism.
\end{lem}
\begin{proof}
We define a homomorphism $\phi:\la\tau_a,\tau_b\ra\rightarrow Sym_{T}$ as $\phi(\tau_c)(d)=\tau_cd$ and a homomorphism $\psi:G\rightarrow Sym_{T}$ as $\psi(c)d=d^c$ for any $c,d\in T$. It follows from Lemma \ref{Mi2} that $\phi(\tau_c)=\psi(c)$. Since the kernels of the homomorphisms $\phi$ and $\psi$ are trivial, $\la\tau_a,\tau_b\ra\simeq G$ and, in particular, $f=\phi\psi^{-1}$ is an isomorphism.
%Определим гомоморфизм $\phi:\la\tau_a,\tau_b\ra\rightarrow Sym_{T}$ следующим образом $\phi(\tau_c)(d)=\tau_cd$ и гомоморфизм $\psi:G\rightarrow Sym_{T}$ следующим образом $\psi(c)d=d^c$ для любых $c,d\in T$. Из леммы \ref{Mi2} следует, что $\phi(\tau_c)=\psi(c)$. Поскольку ядра гомоморфизмов $\phi$ и $\psi$ тривиальны, то $\la\tau_a,\tau_b\ra\simeq G$ и в частносим $f=\phi\psi^{-1}$ является изоморфизмом.  
\end{proof}	

We define a bilinear form on the algebra $A$ and show that for some parameter $\eta$ this form is a left Frobenius form.

\begin{defi}
	We define a bilinear form $(\cdot,\cdot)$ on the algebra $A$ as follows: for any distinct elements $a,b\in T$ put $(a,b)=\eta$ and $(a,a)=1$.
\end{defi}

\begin{lem}\label{twofform}
	If $\eta=-\frac{1}{p-2}$, then the form $(\cdot,\cdot)$ is a left Frobenius form.
	\end{lem}
\begin{proof}
Let $c,d,r\in T$. We prove that $(c*d,r)=(d,c*r)$.
Assume that $c=d$ and $c\neq r$. We have:
$$(c*c,r)=(c,r)=\eta,$$

$$(c,c*r)=(r,r^c+\eta(I(c,r)\setminus r^c))=\eta+\eta^2(n-2)+\eta=\eta.$$
The case when $c=r$ is proved similarly. The case when $c=r$ follows from the symmetry of the form.

Assume that $c,d,r$ are pairwise distinct elements and that $d^c=r$. Note that if $d^c=r$, then $r^c=d$. We have: $$(c*d,r)=(d^c+\eta(I(d,c)\setminus d^c),r)=1+\eta^2*(p-1),$$
$$(d,c*r)=(d,r^c+\eta(I(c,r)\setminus r^c))=1+\eta^2*(p-1)$$
The case when $d^c\neq r$ is checked similarly.

	%Пусть $a,b,c\in T$. Докажем, что $(a*b,c)=(b,a*c)$. 
	%Предположим, что $a=b$ и $a\neq c$. Имеем:
	%$$(a*a,c)=(a,c)=\eta_n,$$
	
	%$$(a,a*c)=(a,c^a+\eta_n(I(a,c)\setminus c^a))=\eta_n+\eta_n^2(n-2)+\eta_n=\eta_n$$
	%Аналогично доказывается случай, когда $a=c$. Случай когда $b=c$ следует из симметричности формы.
	
	%Допустим, что $a,b,c$ попарно различные элементы. Предположим, что $b^a=c$. Заметим, что если $b^a=c$, то $c^a=b$.
	%Имеем: $$(a*b,c)=(b^a+\eta_n(I(a,b)\setminus b^a),с)=1+\eta_n^2*(n-1)$$
	%$$(b,a*c)=(b,c^a+\eta_n(I(a,c)\setminus c^a))=1+\eta_n^2*(n-1)$$
	%Аналогично проверяется случай когда $b^a\neq c$.
\end{proof}

%Пусть $p=5$. Тогда
%\begin{lem}\label{prod25}
%	\begin{enumerate}
%		\item $alp_1*alp_1=(1 - 4 \eta_5)a+16/3 (-1 + \eta_5)z$
%		
%		\item $alp_1*bet_1= bet_1-2 (-1 + \eta_5)bet_2$
%		
%		\item $alp_1*bet_2= -2 (-1 + \eta_5)bet_1- bet_2$
%		
		
%		\item $bet_1*alp_1= (2 - eta)bet_1+(-1 + eta)bet_2$
		
%		\item $bet_2*alp_1= (-1 + eta)bet_1+(-2 + eta)bet_2$
		
%		\item $bet_1*bet_1=-(3 \eta_5/2)z + (2 - \eta_5)/2alp_1$
		
%		\item $bet_2*bet_2=-(3 \eta_5/2)z, x2 -(2 - \eta_5)/2alp_1$
		
%		\item $bet_1bet2=(-1 + eta)/2z, (8/3)(1 - eta)a, (-1 + eta)/2alp_1$
		
%		\item $bet_2bet_1=(1 - eta)/2z+ 8/3 (-1 + eta)a+ 1/2 (-1 + eta)alp_1$
%	\end{enumerate}
%\end{lem}

%\begin{proof}
%	Вычисления легко проверяются при помощи GAP и программы \cite{GorGit}.
%\end{proof}

%\begin{rem}
%Заметим, что при $\eta_5=-1/3$ алгебра $A(G,T,\eta_5)$ порождена двумя осями монстрового типа $(4/3,-4/3)$.
%\end{rem}

\section{Finitely generated case}
In this section we fix: $(G,T)$ is the group of $GM(p)$-type, $\eta\in \F$ and characteristic of $\F$ is good related to $(p,\eta)$. 
%В настоящем параграфе зафиксируем: $(G,T)$ группа строго $p$-транспозиций, $\F$ поле характеристики $0$, $\eta\in \F\setminus\{0,1\}$, $A=A(G,T,\eta)$ .

\begin{lem}\label{inter}
Let $a,b,c\in T$, $X=\lla a,b\rra\cap\lla a,c\rra$. If $c\not \in I(a,b)$, then $X=L(a)$, else $X=\lla a,b\rra=\lla a,c\rra$. 
\end{lem}
\begin{proof}
	Suppose that $c\not \in I(a, b)$. The intersection of the subgroups $\la a, b\ra$ and $\la a, c\ra$ is $\la a \ra$. In particular, $I(a,b)\cap I(a,c)=\{a\}$. Let $x\in X$. Since $I(a,b)$ and $I(a,c)$ are subsets of a basis of the algebra $A$, it follows that $x\in L(I(a,b)\cap I(a,c))$. Therefore, $x\in \lla a\rra$ and $X=L(a)$.
	
	Suppose that $c \in I(a, b)$. In this case, the subgroups $\la a, b\ra$ and $\la a, c\ra$ coincide, and therefore the subalgebras $\lla a,b\rra$ and $\lla a,c\rra$ coincide.  
\end{proof}

For each $a\in T$ we define subset $\{t_i(a)|1\leq i\leq\frac{|T|-1}{p-1}\}=T_a\subseteq T$ such that $\lla a,t_i(a)\rra\cap\lla a,t_j(a)\rra=L(a)$ for distinct $i$ and $j$, and $T=\bigcup_{1\leq i\leq\frac{|T|-1}{p-1}} I(a,t_i(a))$.

\begin{lem}\label{spec}
Let $a\in T$. Then $a$ is a left primitive idempotent and $L_a$ is semisimple linear operator of $A$ with the spectrum $\{1,\lm_1,\lm_2,-\lm_2\}$.  
\end{lem}
\begin{proof}
The element $a$ is idempotent by definition. From Lemma \ref{basis2} it follows that $\lla a, t_i(a)\rra$ has a basis consisting of the eigenvectors of the operator $L_a$.
Let $B(t_i(a))$ be a basis of eigenvectors of $L_a$ of $\lla a, t_i(a)\rra$. Note that $B(t_i(a))\setminus\{a\}$ and $B(t_j(a))\setminus\{a\}$ are linearly independent for $i\neq j$. Thus $B_a=\{a\}\cup (B(t_1(a))\setminus\{a\})\cup...\cup(B(t_{(|T|-1)/(p-1)}(a))\setminus\{a\})$, is a basis of $A$ composed of eigenvectors of $L_a$. In particular, $dim(A_1(L_a))=1$ and spectrum of $L_a$ is $\{1,\lm_1,\lm_2,-\lm_2\}$. 
%Пусть $B(t_i(a))$ базис из собственных векторов линейного оператора $L_a$ пространства $\lla a, t_i(a)\rra$. Заметим, что $B(t_i(a))\setminus\{a\}$ и $B(t_j(a))\setminus\{a\}$ линейно независимы при $i\neq j$. Таким образом $B_a=\{a\}\cup (B(t_1(a))\setminus\{a\})\cup...\cup(B(t_{(|T|-1)/(p-1)}(a))\setminus\{a\})$, является базисом $A$ составленным из собственных векторов $L_a$. В частности $dim(A_1(L_a))=1$.

%Let $z\in A_1(L_a)$. We have $z=\lm a+y_1+...+y_j$, where $y_i\in L(I(a,t_i(a))\setminus\{a\})$. Let $z_i=a*y_i$. We have $z_i\in \lla a,t_i(a)\rra$. Since $\lla a,t_i(a)\rra\cap\lla a,t_j(a)\rra=\L(a)$ for distinct $i$ and $j$, we get $z_i=y_i+\sigma_ia$ for some $\sigma_i\in \F$. Assume that there exists $l$ such that $y_l\not\in L(a)$. Then from Lemma \ref{basis2} it follows that $y_l=\gam_1 a+\gam_2 b+\gam_3 c+\gam_4 d$, where $b, c, d$ are eigenvectors with respect to $L_a$ with eigenvalues $\lm_1,\lm_2,-\lm_2$, respectively, $\gam_1,\gam_2,\gam_3,\gam_4\in \F$. We have $\sigma_la= z_l-y_l=\gam_2(1-\lm_1) b+\gam_3(1-\lm_2) c +\gam_4(1+\lm_2)d$. Since $\lm_1$ and $\lm_2$ are distinct from $1$ and among $\gam_2,\gam_3,\gam_4$ there is at least one element distinct from $0$, we obtain a contradiction.
\end{proof}
\begin{lem}\label{spec2}
Let $a\in T$. Then $a$ is a right primitive idempotent and if $\F$ contains $\frac{p-1}{2}$ roots of $-1$, then $R_a$ is a semisimple operator of $A$.  
\end{lem}
\begin{proof}
   The proof is similar to Lemma \ref{spec} using Lemma \ref{ra2}
\end{proof}

\begin{lem}\label{aaut}
	Let $x,y\in A, a\in T$, then $(x*y)^a=x^a*y^a$.
\end{lem}
\begin{proof}
Since $T$ is a basis of $A$, it is sufficient to prove the assertion of the lemma in situation when $x,y\in T$. Let $x,y\in T$. If $a\in I(x,y)$, then the assertion of the lemma follows from Lemma \ref{Mi2}. We will assume that $a\not\in I(x,y)$.
	%Имеем $(x*y)^a=(y^x+\delta(y^x))^a$
	Since $x^a, y^a\in T$, we have $$x^a*y^a=(y^a)^{x^a}+\dl((y^a)^{x^a})=y^{aaxa}+\dl((y^a)^{x^a})=y^{xa}+\dl((y^a)^{x^a})$$. Note that $$\dl((y^a)^{x^a})=\eta(y^a+x^a+(y^a)^{x^a}+(x^a)^{(y^a)^{x^a}}+...+ y^{xy...xya}-(y^a)^{x^a})=\dl(y^x)^a.$$
	Therefore
	$$x^a*y^a=y^{xa}+\dl((y^a)^{x^a})=(y^x+\dl(y^x))^a=(x*y)^a.$$ 
\end{proof}

\begin{defi}
 Let $A^+_a= \lla a,t_1(a)\rra_a^+\oplus...\oplus\lla a,t_{(|T|-1)/(p-1)}(a)\rra_a^+$, 
		and $A^-_a=\lla a,t_1(a)\rra_a^-\oplus...\oplus\lla a,t_{(|T|-1)/(p-1)}(a)\rra_a^-$. 
\end{defi}

\begin{prop}\label{twograd}
 $(A_a^+,A_a^-)$ is a $\Z_2$-grading of $A$. 
 %С каждым идемпотентом $a\in T$ связана нетривиальная $\Z_2$-градуировка $GA_a=(A_a^+,A_a^-)$ алгебры $A_{\F}(G,T,\eta)$. Если $b\in T\setminus\{a\}$, то $(A_a^+,A_a^-)\neq(A_b^+,A_b^-)$. 	    
\end{prop}	
\begin{proof}
	If $G$ is generated by two elements of $T$, then the proposition follows from Corollary \ref{twogenaxial}. 
    Since $A_a^+\oplus A_a^-$ contains the basis $A$, then $A_a^+\oplus A_a^-=A$. We divide the proof of the proposition into several lemmas.
    
	\begin{lem}\label{xxagen}
		$x+x^a\in A^+_a$
	\end{lem}
	\begin{proof}
	Since $x$ is a sum of elements of $T$, it suffices to show that $b+b^a\in A^+_a$ for $b\in T$. This statement is proved in Lemma \ref{xxa}.
\end{proof}
	
	Let $b,c\in T$ and $c\not \in I(a,b)$, 
    $(a,z, y_1,...,y_k, x_1,...,x_{k+1})=BE(L_a, a, b)$, and
    
    $(a, z', y'_1,...,y'_k, x'_1,...,x'_{k+1})=BE(L_a, a, c)$.
		
	\begin{lem}
	$x_i*x'_j\in A^+_a$
	\end{lem}
	\begin{proof}
	
	%Poskol'ku mnozhestva $\{x_1,...,x_{k+1}\}$ i $\{x'_1,...,x'_{k+1}\}$ ne peresekayutsya, to bez ogranicheniya obshchnosti, my mozhem schitat', chto $i=j=1$. Imeyem $$W=bet_1*bet'_1=(-b+b^a)*(-c+c^a)=b*c+ b^a*c^a-b^a*c-b*c^a.$$ Iz lemmy \ref{aaut} sleduyet, chto $b^a*c^a=(b*c)^a$, $b^a*c=(b*c^a)^a$. % $$b*c^a=c^{ab}+\eta\sum_{x\in I(b,c^a)\setminus c^{ab}}x$$ $$W=b*c+(b*c)^a-b*c^a -(b*c^a)^a$$ Pust' $x=b*c-b*c^a$. Sledovatel'no $W=x+x^a$. Iz lemmy \ref{xxagen} sleduyet, chto $W\in A^+$.

	Since the sets $\{x_1,...,x_{k+1}\}$ and $\{x'_1,...,x'_{k+1}\}$ are disjoint, without loss of generality, we can assume that $i=j=1$.
	We have
	$$W=x_1*x'_1=(-b+b^a)*(-c+c^a)=b*c+ b^a*c^a-b^a*c-b*c^a.$$
	From Lemma \ref{aaut} it follows that $b^a*c^a=(b*c)^a$ and $b^a*c=(b*c^a)^a$. We have
	
	% $$b*c^a=c^{ab}+\eta\sum_{x\in I(b,c^a)\setminus c^{ab}}x$$
	
	$$W=b*c+(b*c)^a-b*c^a -(b*c^a)^a$$
	Let $x=b*c-b*c^a$. Therefore, $W=x+x^a$.
	From Lemma \ref{xxagen} it follows that $W\in A^+$.
	
	%Поскольку множества $\{x_1,...,x_{k+1}\}$ и $\{x'_1,...,x'_{k+1}\}$ не пересекаются, то без ограничения общности, мы можем считать, что $i=j=1$. 
	%Имеем 
	%$$W=bet_1*bet'_1=(-b+b^a)*(-c+c^a)=b*c+ b^a*c^a-b^a*c-b*c^a.$$
	%Из леммы \ref{aaut} следует, что $b^a*c^a=(b*c)^a$, $b^a*c=(b*c^a)^a$.
	
%	$$b*c^a=c^{ab}+\eta\sum_{x\in I(b,c^a)\setminus c^{ab}}x$$
		
	%$$W=b*c+(b*c)^a-b*c^a -(b*c^a)^a$$
	%Пусть $x=b*c-b*c^a$. Следовательно $W=x+x^a$.
	%Из леммы \ref{xxagen} следует, что $W\in A^+$.
	
\end{proof}
	
	Thus, the product of two elements of $A_a^-$ lies in $A_a^+$.
	
	Let's verify that the product of an element of $A^-_a$ and an element of $A^+_a$ lies in $A^-_a$. Since $A^-_a$ is a eigenspace with respect to the operator $L_a$, $a*x$ lies in $A^-_a$ for any element $x\in A^-_a$.
	
	\begin{lem}\label{xmx}
		Let $x\in A$. We have $x-x^a\in A^-_a$.
	\end{lem}
	\begin{proof}
		Since $x$ is a linear combination of elements of $T$, to prove the statement of the lemma it suffices to show that $b-b^a\in A^-_a$, where $b\in T$. This statement proved in Lemma \ref{Mi2}.
		%Поскольку $x$ является линейной комбинацией элементов из $T$, для доказательства утверждения достаточно показать, что $b-b^a\in A^-$, где $b\in T$. Это утверждение следует из Леммы \ref{Mi2}.
	\end{proof}
	\begin{lem}\label{betia}
		Let $x\in A^-_a$. We have $x*a\in A^-_a$.
	\end{lem}
	\begin{proof}
The space $A_a^-$ has a basis of eigenvectors $L_a$. We can assume that $x$ is an eigenvector of $L_a$ with eigenvalue $-\lm_2$ lying in the subalgebra $\lla a,b\rra$ for some $b\in T$.
The assertion of the lemma follows from Corollary \ref{twogenaxial}.
%	Пространство $A_a^-$ имеет базис из собственных векторов $L_a$. Можно считать, что $x$ это собственный вектор оператора $L_a$ с собственным числом $-\lm_2$ лежащий в подалгебре $\lla a,b\rra$ для некоторого $b\in T$.
%Утверждение леммы следует из \ref{twogenaxial}.
\end{proof}

	\begin{lem}\label{cd}
	Let $c,d\in T$. We have $(c+c^a)*(d-d^a)\in A^-_a$.
	\end{lem}
	\begin{proof}
	We have $$W=(c+c^a)*(d-d^a)=a*d-c^a*d^a+c^a*d-c*d^a.$$
From Lemma \ref{aaut} it follows that $c^a*d^a=(c*d)^a$ and $c*d^a=(c^ad)^a$.
From Lemma \ref{xmx} it follows that $W\in A^-_a$.

	%Из леммы \ref{aaut} следует, что $c^a*d^a=(c*d)^a$ и $c*d^a=(c^ad)^a$.
	% Из леммы \ref{xmx} следует, что $W\in A^-$. 	
	\end{proof}
	
	\begin{lem}\label{dc}
	Let $c,d\in T$. We have $(c-c^a)*(d+d^c)\in A^-_a$.
	\end{lem}
	\begin{proof}
The proof is similar to the proof of Lemma \ref{cd}.
%	Доказательство аналогично доказательству леммы \ref{cd}.
	\end{proof}

	\begin{lem}\label{zbet}
	We have $z*x'_i, x'_i*z\in A^-_a$.
	\end{lem}
	\begin{proof}
Note that $z$ is linearly expressible in terms of $a$ and elements of the form $x+x^a$. Thus, the assertion of the lemma follows from Lemmas \ref{cd} and \ref{dc}.

%		Заметим, что $z$ линейно выражается через $a$ и элементы вида $x+x^a$. Таким образом, утверждение леммы следует из лемм \ref{cd} и \ref{dc}. 
	\end{proof}
	
	\begin{lem}\label{alpbet}
	We have $x_i*y'_j, y_i*x'_j\in A^-_a$.
    \end{lem}	 
	\begin{proof}
		Note that $y_i$ is linearly expressible in sum of two elements of the form $x+x^a$. The assertion of the lemma follows from Lemmas \ref{cd} and \ref{dc}.
		%Заметим, что $y_i$ разлагается в сумму двух элементов вида $x+x^a$. Из лемм \ref{cd} и \ref{dc} следует утверждение леммы.
	\end{proof} 
Thus, from Lemmas \ref{betia}, \ref{zbet} and \ref{alpbet} it follows that the products $x*y$ and $y*x$ lie in $A_a^- $, where $x\in A_a^+$ and $y_a\in A^-_a$.
%	 Таким образом из лемм \ref{betia}, \ref{zbet} и \ref{alpbet} следует,что произведения $x*y,y*x\in A_a^- $, где $x\in A_a^+$ и $y_a\in A^-$.
	 
	 We show that $A_a^+$ is a subalgebra. Note that $A_a^+$ has a basis consisting of $a$ and vectors of the form $l+l^a$, where $l\in T$. 
	\begin{lem}\label{pp}
	 	Let $x,y\in T$. Then $(x+x^a)*(y+y^a)\in A^+_a$.
	\end{lem} 
	\begin{proof}
We have $$(x+x^a)*(y+y^a)=x*y+x^a*y^a+x^a*y+x*y^a$$.
The assertion of the lemma follows from Lemma \ref{xxagen}.
%		Имеем $$(x+x^a)*(y+y^a)=x*y+x^a*y^a+x^a*y+x*y^a$$.
%	 	Из леммы \ref{xxagen} следует утверждение леммы.
	\end{proof}

\begin{lem}\label{pa}
	Let $x\in A^+_a$. We have $x*a\in A^+_a$.
\end{lem}
\begin{proof}
We can assume that $x$ lies in a $2$-generated subalgebra. The assertion of the lemma follows from Lemma \ref{prod1}.
%Можно считать, что $x$ лежит в $2$-порожденной подалгебре. Утверждение леммы следует из леммы \ref{prod1}. 
\end{proof}

Thus, it is proved that $(A^+_a,A^-_a)$ is a $\Z_2$-grading of the algebra $A$.
\end{proof}

For any $a\in T$ we will write $\tau_a$ instead of $\tau(A^+_a,A^-a)$. Since $(A^+_a,A^-a)$ is a $\Z_2$-grading, $\tau_a$ is an automorphism of $A$.

\begin{prop}\label{AMia}
The group $Miy(A,T)$ is isomorphic to $G/Z(G)$.
\end{prop}
\begin{proof}
We define a homomorphism $\phi:G \rightarrow Sym(T)$ as the action by conjugation on the set $T$. Clearly, $\phi(G)$ is isomorphic to $G/C_G(T)$. Since $T$ contains a generating set of $G$, we have $C_G(T)=Z(G)$. Since $a,b$ and $\tau_ab$ lie in the subalgebra $\lla a,b\rra$, it follows from Lemma \ref{Mi2} that $\tau_ab=b^a$. Consequently, $\phi(a)= \tau_a$.

%Определим гомоморфизм из $\phi:G \rightarrow Sym(T)$ как действие сопряжением на множестве $T$. Ясно, что $\phi(G)$ изоморфна $G/C_G(T)$. Поскольку $T$ содержит порождающее множество группы $G$, то $C_G(T)=Z(G)$. Поскольку $a,b$ и $\tau_a(b)$ лежат в подалгебре $\lla a,b\rra$, то из Леммы \ref{Mi2} следует $\tau_a(b)=b^a$. Следовательно $\phi(a)= \tau_a$. 	
\end{proof}

%If $b\in T\setminus\{a\}$, then $(A_a^+,A_a^-)\neq(A_b^+,A_b^-)$.

\begin{defi}
We define a symmetric bilinear form $(\cdot,\cdot)$ on the algebra $A$ by the following rule: for any distinct elements $a,b\in T$ we set $(a,b)=\eta$ and $(a,a)=1$.
\end{defi}

\begin{lem}
If $\eta=-\frac{1}{n-2}$, then $(\cdot,\cdot)$ is a left Frobenius form.
%	Если $\eta=-\frac{1}{n-2}$, то $(\cdot,\cdot)$ является левой формой Фробениуса.
\end{lem}
\begin{proof}
Since the form is bilinear, it suffices to show that $(a*b,c)=(b,a*c)$ for any three elements $a,b,c\in T$. Note that if $a\in \lla b,c\rra$, then the assertion of the lemma follows from Lemma \ref{twofform}.
%	Поскольку форма билинейно, то достаточно показать, что $(a*b,c)=(b,a*c)$ для любых трех элементов $a,b,c\in T$. Заметим,что если $a,b,c$ лежат в $2$-порожденной подгруппе группы $G$, то утверждение леммы следует из леммы \ref{twofform}. 
	We will assume that $a\not\in \lla b,c\rra$. We have: $$(a*b,c)=(b^a+\dl(a^b),c)=\eta+\eta^2(n-1),$$
	$$(b,a*c)=(b,c^a+\dl(a^c)))=\eta+\eta^2(n-1).$$
	 Thus, $(a*b,c)=(b,a*c)$.
\end{proof}
\begin{defi}
Let $f(\cdot,\cdot)$ be a symmetric form of an algebra $A$. The radical of a form $(\cdot,\cdot)$ is $Rad_f(A)=\{x\in A|(x,y)=0$ for each $y\in A\}$. In cases when it is clear which bilinear form we are talking about, we will simply write $Rad(A)$.
%Пусть $(\cdot,\cdot)$ форма алгебры $A$. Радикалом формы называется множество элементов $\R(A)\leq A$ такое, что $(x,a)=0$ для любого $a\in A$.	
\end{defi}
\begin{lem}\label{Tr}
$Rad(A)=0$.
\end{lem}
\begin{proof}

%Пусть $x$ элемент радикала. В базисе $B_a$ имеем $x=(\al_1,...,\al_n)$. По определению формы имеем $$(x,a)=\al_1+\eta(\al_2+...+\al_n)=0$$
%Данное равенство верно для любого элемента из $T$. Таким образом $x$ является решением однородной системы линейных уравнений с матрицей $$X=\begin{pmatrix}
Let $x\in Rad(A)$. We have $x=(\al_1,...,\al_n)*T$. By the definition of the form, we have $$(x,a)=\al_1+\eta(\al_2+...+\al_n)=0$$
This equality is true for any element $a\in T$. Thus, $x$ is a solution to a homogeneous system of linear equations with the matrix $$X=\begin{pmatrix}

	1    & \eta   &\cdots&\eta \\
	\eta & 1 &\cdots& \eta \\ 
	\eta & \eta &\cdots& \eta\\
	\vdots&\vdots& \ddots & \vdots\\
	\eta & \eta&\cdots& 1&\\
\end{pmatrix}.$$  

Solving a homogeneous system with matrix $X$, we find that there are no non-zero solutions.
\end{proof}

\begin{prop}\label{noideal}
	The algebra $A$ does not contain right ideals.
\end{prop}
\begin{proof}
	
	Suppose that $A$ contains a non-trivial right ideal $I$. We prove that $\tau_a(I)=I$ for any $a\in T$.
	
	\begin{lem}\label{orb}
	If $a\in T$, then $\tau_a(I)\leq I$.
	\end{lem}
	\begin{proof}
		%Пусть $s\in A$. Тогда $s_1=s-\frac{a*s}{\lm_2}$ имеет тривиальную проекцию на $A_{\lm_2}(a)$. Элемент $s_2=s_1-\frac{a*s_1}{\lm_1}$ имеет тривиальную проекцию на $A_{\lm_1}(a)$. Имеем $s_3=\frac{\lm_1}{2(\lm_1+\lm+2)(\lm_2+1)}(s_2-a*s_2)$ равен проекции элемента $s$ на подпространство $A_a^-$. Таким образом $\tau_a(s)=s-2s_3$. В частности, если $s$ является элементом правого идеала $I$, то и $\tau_a(s)\in I$. 
		 Let $s\in A$. The element $s_1=s-\frac{a*s}{\lm_2}$ has the trivial projection on $A_{\lm_2}(a)$. The element $s_2=s_1-\frac{a*s_1}{\lm_1}$ has the trivial projection on $A_{\lm_1}(a)$. We have $s_3=\frac{\lm_1}{2(\lm_1+\lm+2)(\lm_2+1)}(s_2-a*s_2)$ equal to the projection of the element $s$ on the subspace $A_a^-$. Thus $\tau_a(s)=s-2s_3$. In particular, if $s$ is an element of the right ideal $I$, then $\tau_a(s)\in I$. 		
%		В доказательстве леммы 3.10 \cite{KMS} не используется коммутативность алгебры, по этому данное утверждение является следствием 3.11 \cite{KMS}.
	\end{proof}
	\begin{lem}\label{IdAx}
		A non-trivial ideal does not contain elements from $T$.
	\end{lem}
	\begin{proof}
Let $a\in T\cap I$. Since the orbit of $a$ under the action of $Mia(A,T)$ is $T$, then applying lemma \ref{orb} we obtain that $T\subseteq I$. Since $T$ generates the algebra $A$, it follows that $I=A$.	
\end{proof}
	\begin{lem}\label{fin}
The algebra $A$ does not contain non-trivial right ideals.         

\end{lem}
	\begin{proof}
Let $I$ be a right ideal. We prove that $I$ lies in the radical of the form $(\cdot,\cdot)$. For this, it suffices to show that $(a,w)=0$ for any $w\in I$ and any $a\in T$. Put $\sigma a$ is the projection of $w$ onto the subspace $L(a)$, for some $\sigma\in \F$. Then the linear combination of words $w,a*w,a*(a*w),a*(a*(a*w))$ with properly chosen parameters is $\sigma a$. In particular, if $\sigma\neq 0$, then $a\in I$; a contradiction with Lemma \ref{IdAx}. Since $(a,w)$ is a projection onto $a$, we have $(a,w)=\sigma=0$. Thus, $w$ is an element of the radical of the form $(\cdot,\cdot)$. Since the radical of the form is trivial see Lemma \ref{Tr}, $I$ is trivial.
%		Пусть $I$ идеал. Докажем, что $I$ лежит в радикале формы $(\cdot,\cdot)$. Для этого, достаточно показать, что для любого $w\in I$ и любого $a\in T$ верно $(a,w)=0$. Предположим, что проекция $w_a$ элемента $w$ на подпространство $\lla a\rra$ не тривиальна. Тогда линейная комбинация слов $aw,a(aw),a(a(aw)$ при правильно подобранных параметрах равна $a$. В частности $a\in I$; противоречие с Леммой \ref{IdAx}. Поскольку $(a,w)$ является проекцией на ось $a$, то $(a,w)=w_a=0$. Таким образом $w$ элемент радикала формы $(\cdot,\cdot)$. Поскольку радикал формы тривиален \ref{Tr}, то и $I$ тривиален.   
	\end{proof}
	Lemma \ref{fin} completes the proof of the proposition.
\end{proof}

Statements 1 and 2 of Theorem 3 are proved in Lemmas \ref{spec} and \ref{spec2}. Statement 3 of Theorem 3 is proved in Proposition \ref{noideal}.
Statement 4 of the Theorem 3 is proved in Propositions \ref{twograd} and \ref{AMia}.

%Утверждение 1 и 2 Теоремы доказанно в леммах \ref{spec} и \ref{spec2}. Утверждение 3 Теоремы доказано в Предложении 6. 

%Утверждение 4 Теоремы доказанно в Предложениях \ref{twograd} и \ref{AMia}
\section{$GM(p,\eta)$-algebras}
In this section, we will define algebras whose definitions are independent of the $GM(p)$-group, but whose Miyamoto group is a $GM(p)$-group. We will also prove that this definition is equivalent to the definition of an $A_{\F}(G,T,\eta)$-type algebra.

Let  $\eta\in \F$, and let the characteristic of $\F$ be good related to $(p,\eta)$. Let $p$ be a prime greater than $2$. We define a $GM(p,\eta)$-type algebra.

\begin{defi}
We say that an $\F$-algebra $A(T)$ with basis $T$ is an algebra of $GM(p,\eta)$-type if the following statements hold:
\begin{enumerate}
	\item every element of $T$ is idempotent;
	\item for any pair of elements $a,b\in T$, $\lla a,b\rra\simeq A_{\F}(\la a, b\ra,I(a,b),\eta)$ is a dihedral algebra of dimension $p$;
	\item for any element $a\in T$, there exists a subset $\Omega_a=\{t_1,...\}\subseteq T$ such that $T=\{a\}\bigcup_{t\in\Omega_a}(I(a,t)\setminus\{a\})$;
	\item for any $a\in T$, the linear mapping $\phi_a:A\rightarrow A$ defined by the rule $\phi_a(b)=b^a$ for any $b\in T$ is an automorphism of $A$.
\end{enumerate} 
\end{defi}	

%	Будем говорить, что $\F$-алгебра $A(T)$ с базисом $T$ является алгеброй $GM(p,\eta)$-типа если выполнены следующие утверждения:
%	\begin{enumerate}
%		\item любой элемент из $T$ является идемпотентом; 
%		\item для любой пары элементов $a,b\in T$ верно $\lla a,b\rra\simeq A_{\F}(\la a, b\ra,I(a,b),\eta)$-диэдральная алгебра размерности $p$;
%		\item для любого элемента $a\in T$ найдется подмножество $\Omega_a=\{t_1,...\}\subseteq T$ такое, что $T=\{a\}\bigcup_{t\in\Omega_a}(I(a,t)\setminus\{a\})$;
%		\item для любого $a\in T$ линейное отображение $\phi_a:A\rightarrow A$ определенное по правилу $\phi_a(b)=b^a$ для любого $b\in T$ является автоморфизмом $A$.
%	\end{enumerate} 
%\end{defi}	

Until the end of this section, we fix some algebra $A=A(T)$ of $GM(p,\eta)$-type.

Note that the definition of algebras $GM(p,\eta)$-type implies that for any group $(G,T)$ of $GM(p)$-type the algebra $A_{\F}(G,T,\eta)$ is an algebra of $GM(p,\eta)$-type.

\begin{defi}
For any element $a\in T$, we define the subspace $$A_a^+=L(\lla a,t\rra_a^+| t\in \Omega_a)$$
$$A_a^-=L(\lla a,t\rra_a^+| t\in \Omega_a)$$
\end{defi}
Clearly, $A=A_a^+\oplus A_a^-$. We prove that $(A_a^+, A_a^-)$ is a $\Z_2$-grading of the algebra $A$.

\begin{prop}\label{Zgrad}
Let $a\in T$, then $(A_a^+, A_a^-)$ is a $\Z_2$-grading of the algebra $A$.
\end{prop}
 \begin{proof} 
 We divide the proof of the proposition into several lemmas.
Prove that for any $x\in A$ the element $x+\phi_a (x)\in A_a^+$.
%Докажим, что для любого $x\in A$ элемент $x+\phi_a (x)\in A_a^+$.    
\begin{lem}\label{apluss}
For any $x\in A$, $x+\phi_a(x)\in A_a^+$ holds.
%Для любого $x\in A$ верно $x+\phi_a(x)\in A_a^+$.
\end{lem}
\begin{proof}
Note that $x$ can be factored as a sum $x=\alpha_1x_1,...,\alpha_kx_k$, where $x_1,...,x_k\in T$. By linearity, we have $x+\phi_a(x)=\alpha_1(x_1+x_1^a)+...+\alpha_1(x_k+x_k^a)$. From Lemma \ref{xxa} it follows that $x_1+x_1^a,...,x_k+x_k^a \in A_a^+$. Thus, $x+\phi_a(x)\in A_a^+$.
%	Заметим, что $x$ раскладывается в виде суммы $x=\alpha_1x_1,...,\alpha_kx_k$, где $x_1,...,x_k\in T$. По линейности имеем $x+\phi_a(x)=\alpha_1(x_1+x_1^a)+...+\alpha_1(x_k+x_k^a)$. Из леммы \ref{xxa} следует, что $x_1+x_1^a,...,x_k+x_k^a \in A_a^+$. Таким образом $x+\phi_a(x)\in A_a^+$.
\end{proof}

Now we can prove that $A_a^+$ is a subalgebra.

\begin{lem}\label{um}
	The subspace $A_a^+$ is a subalgebra.
\end{lem}
\begin{proof}
Let $x\in \lla a, b\rra _a^+$, $y\in \lla a,d\rra_a^+$.
It suffices to show that $x*y$ can be represented as $w+\phi_a(w)$, where $w\in A$. Note that any element of $\lla a, b\rra _a^+$ can be represented as a sum of elements of the form $g+g^a$, where $g\in I(a,b)$. If $\lla a, b\rra=\lla a, d\rra$, then the assertion of the lemma follows from Corollary \ref{twogenaxial}. We can assume that $x=b+b^a, y=d+d^a$. We have:
$$(b+b^a)*(d+d^a)=b*d +b^a*d^a+b^a*d+b*d^a.$$
Since $\phi_a$ is an automorphism and $\phi_a(b)=b^a, \phi_a(d)=d^a$, then $b*d +b^a*d^a+b^a*d+b*d^a=(b*d+b^a*d)+\phi_a(b*d+b^a*d)$. Thus, from Lemma \ref{apluss} the product of any two elements from $A_a^+$ lies in $A_a^+$ and therefore $A_a^+$ is a subalgebra.

%Пусть $x\in \lla a, b\rra _a^+$, $y\in \lla a,d\rra_a^+$. 
%Достаточно показать, что $x*y$ представим в виде $w+\phi_a(w)$, где $w\in A$. Заметим, что любой элемент из $\lla a, b\rra _a^+$ представим в виде суммы элементов вида $g+g^a$ где $g\in I(a,b)$. Если $\lla a, b\rra=\lla a, d\rra$, то утверждение леммы следует из Следствия \ref{twogenaxial}. Можно считать, что $x=b+b^a, y=d+d^a$. Имеем: 
%$$(b+b^a)*(d+d^a)=b*d +b^a*d^a+b^a*d+b*d^a.$$
%Поскольку $\phi_a$ автоморфизм и $\phi_a(b)=b^a, \phi_a(d)=d^a$, то $b*d +b^a*d^a+b^a*d+b*d^a=(b*d+b^a*d)+\phi_a(b*d+b^a*d)$.
\end{proof}

\begin{lem}\label{dois}
Let $x,y\in A_a^-$. Then $x*y\in A_a^+$
%Пусть $x,y\in A_a^-$. Тогда $x*y\in A_a^+$
\end{lem}	
\begin{proof}
	Note that any element of $A_a^-$ can be represented as $w-\phi_a(w)$, where $w\in A$. Therefore, it suffices to show that the element $(b-b^a)*(d-d^a)$, where $b,d\in T$, can be represented as $w+\phi_a(w)$. We have: $$(b-b^a)*(d-d^a)=(b*d-b^a*d)+\phi_a(b*d-b^a*d).$$ Thus, it follows from Lemma \ref{aplus} that the product of any two elements from $A_a^-$ lies in $A_a^+$.
	
%Заметим, что любой элемент из $A_a^-$ представим в виде $w-w^a$, где $w\in w$. По этому, достаточно показать, что элемент $(b-b^a)*(d-d^a)$, где $b,d\in T$, представим в виде $w+\phi_a(w)$. Имеем: $$(b-b^a)*(d-d^a)=(b*d-b^a*d)+\phi_a(b*d-b^a*d).$$  
\end{proof}

\begin{lem}\label{treis}
Let $x\in A_a^-$, $y\in A_a^+$. Then $x*y, y*x\in A_a^-$
\end{lem}	
\begin{proof}
Note that any element of $A_a^-$ can be represented as $w-\phi_a(w)$, where $w\in A$, and any element of $A_a^+$ can be represented as $w'+\phi_a(w')$, where $w'\in A$. We will show that the element $(b-b^a)*(d+d^a)$, where $b,d\in T$, can be represented as $w-\phi_a(w)$. We have: $$(b-b^a)*(d+d^a)=(b*d-b^a*d)-\phi_a(b*d-b^a*d).$$
We will show that the element $(b+b^a)*(d-d^a)$, where $b,d\in T$, can be represented as $w-\phi_a(w)$. We have: $$(b+b^a)*(d-d^a)=(b*d-b*d^a)-\phi_a(b*d-b*d^a).$$
\end{proof}

From Lemmas \ref{um}, \ref{dois}, \ref{treis} the assertion of the proposition follows.

\end{proof}

For any $a\in T$, we define a linear mapping $\tau_a: A\rightarrow A$ as follows: $\tau_a(x)=x, \tau_a(y)=-y$, where $x\in A_a^+, y\in A_a^-$.
Proposition \ref{Zgrad} implies that $(A_a^+,A_a^-)$ is a $Z_2$-graded of $A$. Consequently, $\tau_a$ is an automorphism of $A$.

Note that for any $a,b\in T$, Lemma \ref{Mi2} implies that the restrictions of $\phi_a$ and $\tau_a$ to the subalgebra $\lla a,b\rra$ coincide. Since $T$ is a basis of $A$, $\phi_a$ and $\tau_a$ are equal on the whole algebra $A$.

Let us prove that the group $G=\la \phi_a|a\in T\ra$ is a $GM(p)$-group.

\begin{lem}\label{aut}
For any $\theta \in G$ and any $b\in T$, we have $\theta(b)\in T$.
\end{lem}
\begin{proof}
	We have $\theta=\phi_{g_1}...\phi_{g_k}$, where $g_1,...,g_k\in T$. It is enough to show that  $\phi_{g_1}(b)\in T$. Note that $b, g_1\in \lla b,g_1\rra$, and $\phi_{g_1}(b)=b^{g_1}\in I(b,g_1)\subseteq T$.
\end{proof}

The following statement is a consequence of Lemma \ref{tauaut}.
\begin{lem}\label{tauaut2}
We have $\phi_a\phi_b\phi_a=\phi_{\phi_a(b)}$.	
\end{lem}

\begin{lem}\label{MiyGM}
Let $a,b\in T$. Then $|\phi_a\phi_b|=p$.  
\end{lem}
\begin{proof}
	Note that the restriction of the operator $\phi_a\phi_b$ to the subalgebra $\lla a, b\rra$ has order $p$. Therefore, $|\phi_a\phi_b|$ is divisible by $p$.
	Suppose that $|\phi_a\phi_b|=n>p$. Let $D=\la \phi_a,\phi_b\ra$. Then Lemma \ref{aut} implies that $a^g\in T$ for any $g\in D$. Lemma \ref{tauaut2} implies that $|a^D|=n$. Thus $|I(a,b)|=n$, but the definition of the algebra $\lla a,b\rra$ implies $|I(a,b)|=p$; a contradiction.
	
	%Предположим, что $|\phi_a\phi_b|=n>p$. Пусть $D=\la \phi_a,\phi_b\ra$. Тогда из леммы \ref{aut} следует, что $a^g\in T$ для любого $g\in D$. Из леммы \ref{tauaut2} следует, что $|a^D|=n$. Таким образом $|I(a,b)|=n$, но из определению алгебры $\lla a,b\rra$ следует $|I(a,b)|=p$; противоречие.
\end{proof}	
From Lemma \ref{MiyGM} follows Theorem 4.
\bigskip

\noindent\textbf{Theorem 4.}
\textit{Let $A(T)$ be an algebra of $GM(p,\eta)$-type. Then $M=Miy(A,T)$ is a $GM(p)$-group and $A(T)\cong A_{\F}(M,T,\eta)$.	
}
\begin{proof}
Lemma \ref{MiyGM} implies that $M$ is a $GM(p)$-group. Let $B=A_{\F}(M,T,\eta)$. Multiplication in the algebra $B$ will be denoted by $\circ$. Since $T$ is a basis in $A$ and $B$, it suffices to show that $a*b=a\circ b$ for any $a,b\in T$. This assertion follows from the fact that $\lla a,b\rra_*=\lla a,b\rra_{\circ}$ by definition.
\end{proof}

      \bigskip

Ilya~B. Gorshkov

Sobolev Institute of Mathematics,

Novosibirsk, Russia,

E-mail address: ilygor8@gmail.com

\end{document}